\documentclass[11pt]{article}
\usepackage[tbtags]{amsmath}
\usepackage{amsfonts,amssymb,mathrsfs,amscd}
\usepackage{amsthm}

\usepackage{graphicx}

\usepackage{hyperref}

\numberwithin{equation}{section}

\theoremstyle{plain}
\newtheorem{theorem}{Theorem}
\newtheorem{lemma}{Lemma}[section]
\newtheorem{propos}[lemma]{Proposition} 
\newtheorem{corollary}[lemma]{Corollary} 
\newtheorem{onecorollary}[lemma]{Corollary} 
\theoremstyle{definition}
\newtheorem{definition}{Definition}[section]
\newtheorem{proofmain}{Proof of Theorem 1}
\newtheorem{remark}[lemma]{Remark} 

\newtheorem{opquest}{Question}
\renewcommand{\Re}{\operatorname{Re}}

\newcommand{\const}{\mathrm{const}}
\newcommand{\RR}{\mathbb{R}}
\newcommand{\Rp}{\RR_+}

\newcommand{\CC}{\mathbb{C}}
\newcommand{\ZZ}{\mathbb{Z}}

\newcommand{\ba}[1]{\begin{array}{#1}}
\newcommand{\ea}{\end{array}}
\newcommand{\beq}[1]{\begin{equation}\label{#1}}
\newcommand{\eeq}{\end{equation}}

\newcommand{\supp}{\mathrm{supp}\,}

\newcommand{\rbr}[1]{\left(#1\right)}
\newcommand{\fbr}[1]{\left\{#1\right\}}
\newcommand{\abr}[1]{\left|#1\right|}
\newcommand{\weakto}{\rightharpoonup}
\newcommand{\wlim}{\mathop{\mathrm{w\mbox{-}lim}}}
\newcommand{\diam}{D}
\newcommand{\dsupp}{\supp}
\newcommand{\indf}[1]{I_{#1}}
\newcommand{\keps}{k_{\rho}}
\newcommand{\Keps}{K_{\rho}}
\newcommand{\pow}[2]{{#1}^{\langle #2\rangle}}
\newcommand{\shift}{\hat}
\newcommand{\Seq}[1]{\mathrm{#1}}

\newcommand{\Lt}{\mathcal{L}} 
\newcommand{\Ft}{\mathcal{F}} 
\newcommand{\numN}{N^\circ}

\newcommand{\ga}{\gamma}

\newcommand{\eps}{\varepsilon}

\newcommand{\norm}[1]{\left\|#1\right\|}

\def\rem#1{}

\renewcommand{\text}[1]{\mbox{\rm #1}}

\newcounter{refnote}

\newcommand{\eref}[1]{${}^{(\mbox{\it\ref{#1}})}$}
\newcommand{\ereftxt}[1]{({\it\ref{#1}})}
\newcommand{\enote}[1]{\refstepcounter{refnote}\label{#1}%
           ({\it\ref{#1}})}

\begin{document}

\author{{Gleb Kalachev}%
\footnote{
{Moscow State University, Russia;}
{email: gleb.kalachev@yandex.ru}}
\ and
{Sergey Sadov}%
\footnote{{Moscow, Russia};
{email: serge.sadov@gmail.com}}
}

\title{On maximizers of convolution operators in $L_p$ spaces} 





\maketitle

\begin{abstract}
We consider a convolution operator in
$\RR^d$ with kernel in $L_q$ acting from
$L_p$ to $L_s$, where $1/p+1/q=1+1/s$.
The main theorem states that if $1<q,p,s<\infty$, then there exists an $L_p$ function of unit norm  
on which the $s$-norm of the convolution is attained.
A number of questions, 
related to the statement and proof of the main theorem, are discussed.
Also the problem of computing best constants in the Hausdorff-Young inequality for the Laplace transform, which
prompted this research, is considered.


\end{abstract}

\markright{Maximizers of convolution operators}

\section{Introduction}
\label{sec:intro}

Let $L_p(\RR^d)$ denote the Lebesgue space of measurable complex-valued functions with norm $\|f\|_p=(\int |f|^p)^{1/p}$, where $1\leq p<\infty$,
or with norm $\|f\|_\infty=\sup\{a>0\,|\, |f(x)|\leq a \;\text{a.e.}\}$. Throughout, $p'=(1-1/p)^{-1}$ denotes the conjugate exponent.
We consider a convolution operator $K_k:\, f\mapsto k*f$ with kernel  $k\in L_{q}(\RR^d)$, 
$$
K_k f(x)=\int k(x-y) f(y)\,dy.
$$
As long as there is no ambiguity, we use shorthand notation:
$L_p$ instead of $L_p(\RR^d)$, $\int f$ instead of $\int f(x)\,dx$, and $K$ instead of $K_k$. (If $k_\lambda(\cdot)$ is a family of kernels depending on a parameter, we write $K_\lambda$ instead of $K_{k_\lambda}$.) In formulations and proofs of statements it is assumed that the dimension $d$ and the kernel $k$ are fixed.

Let $1\leq p,q,r\leq\infty$. If $k\in L_q$ and the relation
\beq{pqr}
\frac{1}{p}+\frac{1}{q}+\frac{1}{r}=2
\eeq
holds, 
then the operator $K$ acts boundedly from $L_p$ to $L_{r'}$ 
and its norm (to be called the $(p,r)$-norm%
%
\footnote{Emphasizing that $\|K\|_{p,r}$ is the extremum of the symmetric bilinear form
$\|K\|_{p,r}=\|K\|_{r,p}=\sup\big|\int k(x+y)f(x)g(y)\,dx\,dy\big|; \,\norm{f}_p=\norm{g}_r=1$.
}%
) has an upper bound given by Young's inequality 
$\|K\|_{p,r}\leq \|k\|_q$.
A function $f\in L_p$ is called a \emph{maximizer}\ of the convolution operator $K$ with respect to the pair of exponents $(p,r)$ if $\|f\|_p=1$ and
$\|k*f\|_{r'}=\|K\|_{p,r}$.

The main result of the paper is the theorem on existence of a maximizer.

\begin{theorem}
\label{thm:minimizer_exists}
Let $1<p,q,r<\infty$ and the relation \eqref{pqr} holds. Then for any kernel $k\in L_q$  there exists a maximizer of the operator $K$ 
 with respect to the pair of exponents $(p,r)$.
\end{theorem}

In a narrow sense, the only predecessor of this result that we are aware of is the paper by M.\,Pearson \cite{Pearson_1999}.
In it, the existence of a maximizer is proved under the following assumptions:
the function $k$ is radially symmetric, nonnegative, decreasing away from the origin; besides, an ``extra room of integrability'',
$k\in L_{q+\eps}\cap L_{q-\eps}$, is required. 

In a wider sense, our paper's context is related, on the one hand
(from the motivation side),
to the business of sharp constant in analytical inequalities, and on the other hand, to techniques of proving the existence of extremizers in variational problems with non-compact groups of symmetries.

The starting point of this work was computation of the norms of
the Laplace transorm as an operator from $L_p$ to 
 $L_{p'}$ ($1\leq p\leq 2$) on $\RR_+$ (that is, of sharp constants in G.H.\,Hardy's \cite{Hardy1933} inequality),
reported here in Section~\ref{sec:Laplace-norm}.
An equivalent problem is to compute the corresponding norms of the convolution operators on $\RR^1$ with kernels $h_p(x)=\exp(x/p'-e^{x})$.
Hardy's original estimate can be improved by making use of W.~Beckner's sharp form \cite{Beckner1975b} of Young's inequality%
\eref{com:Beckner}.%
\footnote{The labels \eref{com:Beckner}, \eref{com:Christ} etc.\ refer to the comments \ereftxt{com:Beckner}, \ereftxt{com:Christ}  etc.\ in Section\,\ref{sec:comments}.}
\beq{beckC}
 \|K\|_{p,r}\leq (A_p A_q A_{r})^d\,\|k\|_q ,
\quad\text{where}
\quad
A_m=\left(\frac{m^{1/m}}{{m'}^{1/m'}}\right)^{1/2},
\eeq
The so improved estimate --- see \eqref{EricLp} --- is found in the 2005's M.\,Sc.\ Thesis of E.~Setterqvist \cite[Theorem~2.2]{Setterqvist2005}.

The equality in Beckner's inequality \eqref{beckC} takes place only in the case of a Gaussian kernel $k$. 
Furhter analytical enhancement of the estimate \eqref{EricLp}
should in principle be possible by using recent subtle results   
of M.\,Christ \cite{Christ_2011,Christ_2014}\eref{com:Christ}.

Our numerical results on the norms $\|K_{h_p}\|_{p,p}$ (lacking full justification) allow one to make judgement about comparative strength of the available analytical estimates (see Section\,\ref{sec:Laplace-norm}). For $1<p<2$, the 
existence of maximizers was first observed experimentally. (If $p=2$, there is no maximizer as one can easily see.)
Since the kernel $h_p$ is not symmetric, Pearson's theorem is not applicable. The question about
possible weaker conditions sufficient for the existence of a maximizer naturally presented itself. 
Somewhat surprisingly, it turned out that no artificial conditions are needed at all.

A difficulty in proving the existence of a maximizer of a convolution operator, as well as in similar situations, owes to the problem's invariance
with respect to a non-compact group of transformations
 (the additive group $\RR^d$ of translations, in our case). 
A natural attempt is to begin with some normalized maximizing sequence  
 $(f_n)$, $\|f_n\|_p=1$,  $\|Kf_n\|_{r'}\to\|K\|$, and, referring to the Banach-Alaoglu theorem, to find a weakly converging subsequence $(f_{n(m)})$. 
There is no chance to prove strong convergence of the subsequence $(f_{n(m)})$ since it may ``run away'' to infinity, thus weakly converging to zero.  

However one can take shifted functions $\tilde f_n=f_n(\cdot-a_n)$ hoping to chose the shifts $a_n$ so as to 
obtain a relatively compact sequence $(\tilde f_{n})$. (In terms of inequalities it amounts to establishing suitable 
uniform estimates). If this idea works out, a weakly converging subsequence of the shifted sequence will converge strongly and its limit
will be a maximizer. 

Person's proof exploits the fact that
in the case of a radially symmetric kernel the functions $\tilde f_n$ can also be made radially symmetric. This is due to 
M.\,Riesz's inequality for nondecreasing rearrangements \cite[Theorem~3.7]{Lieb-Loss_1998}. We do not know of any substitute (analog, generalization)  
for Riesz's inequality in the case of a non-symmetric kernel,
which would allow a generalization of Pearson's argument. Availability of full analytical control of the kernel in all directions (like in the case of $h_p$) does not help. Our approach is completely different. 
Absent a natural reference point, we develop an intrinsic way to describe localization of near-maximizers.
 

The scheme and some elements of our proof
exhibit clear similarities with the concentration compactness method of P.\,L.~Lions \cite{Lions_1984a}. Moreover, the class of variational problems described in the introductory section of Lions' paper \cite{Lions_1984a} includes the problem of finding the $(p,r)$-norm of a convolution operator as one of the simplest representatives of the class.%
\eref{com:Lions-intro-problem}.


For this reason one may wonder at 
the absence (to the best of our knowledge) of 
Theorem~\ref{thm:minimizer_exists} in the existing literature.\eref{com:Tao-toy-problem}
Note though that the proof presented here does not depend on Lions' work, whether directly or indirectly. 
  
Let us outline the paper's structure. The proof of Theorem~\ref{thm:minimizer_exists} is given in Sections\,\ref{sec:proof-general}--\ref{sec:lemmas}.
In Section\,\ref{sec:proof-general} we introduce relevant terminology and describe the proof ``in the large''.
Properties required at the steps of the proof
are mentioned and references to the places where they are
treated in detail are given. 
Detailed formulations and proofs of the required properties as well as auxiliary intermediate results are contained in Sections 
~\ref{sec:non-spreading}~and~\ref{sec:lemmas}. 
The key Lemmas~\ref{lemma:controlled_support}--\ref{lemma:concentration_property} of Section~\ref{sec:non-spreading} provide
uniform control on the size of 
``near-supports'' of member functions of a maximizing sequence,
to exclude a possibility of diffusion. (It helps to always ask in the course of the proof, where it fails when $q=1$, $p=r=2$, in which case there is no maximizer if $k(x)>0$). 
The lemmas of Section~\ref{sec:lemmas} deal with compactness and shifts. 

In Section\,\ref{sec:discussion},
a number of diverse results related to 
Theorem~\ref{thm:minimizer_exists} and its proof
are treated. In particular, we discuss
the cases of exponents
$1$ and $\infty$, the equation satisfied by a maximizer, the lower bounds for convolution operators, etc.
A survey of the results is given in
Subsection\,\ref{ssec:discussion-survey}. 
 
Section\,\ref{sec:Laplace-norm} is devoted to
computation of the norms of the Laplace transform
from $L_p(\RR_+)$ to $L_{p'}(\RR_+)$;
the equivalent problem being, as mentioned before, the calculation of the norms of  convolutions with kernels $h_p(x)$ as operators from $L_p(\RR)$ to $L_{p'}(\RR)$. 
The obtained numerical results are compared with several analytical estimates.
The numerical method used is straightforward; however, its convergence remains an empirical fact.

The short Section\,\ref{sec:open-problems} 
contains some open questions and conjectures
that lie close to the paper's contents.

In the final Section\,\ref{sec:comments}, 
bibliographical and terminological comments
are gathered.

\smallskip\noindent
\centerline{---------------------}

\smallskip
The following
equivalent forms of the relation
\eqref{pqr} will be used repeatedly as convenient. 
Let $\{x,y,z\}=
\{p,q,r\}$ in any order. Then
$$
 \frac{1}{x}+\frac{1}{y}=1+\frac{1}{z'},
 \qquad
 \frac{1}{x'}+\frac{1}{y'}=\frac{1}{z},
\qquad
x'\geq z.
$$
Consequently,
$
\min(p',q',r')\geq \max(p,q,r),
$
with equality if and only if $1\in \{p,q,r\}$.

\section{Preliminaries and the proof in the large}
\label{sec:proof-general}

The exponents $q,p,r$ and the convolution kernel $k\in L_q$ are assumed fixed. 

By definition of the norm of the operator $K$, for any $\eps>0$ there exists a function $f_{\eps}\in L_p$ such that
$\norm{f_\eps}_p=1$ and $\norm{Kf_\eps}_{r'}\ge \norm{K}_{p,r}(1-\eps)$. 
Any such function will be called an {\em $\eps$-maximizer}.

\begin{definition}
\label{def:maximizing_sequence}
 A sequence $(f_n)$ of norm one functions from $L_p$ is a {\em maximizing sequence}\ (for the operator $K$)
 if $\norm{Kf_n}_{r'}\to \norm{K}_{p,r}$ as $n\to\infty$.
\end{definition}

The proof of Theorem~\ref{thm:minimizer_exists} aims, quite obviously, at constructing a maximizing sequence for the operator $K$ that converges in norm. 
The limit will be a norm one function and a maximizer.
We will start out with an arbitrarily chosen maximizing sequence,  apply to its members certain ``improving operators'' and shifts (translations), and finally select a suitable, strongly convergent subsequence.

\smallskip
Note a few trivial but important properties.

\smallskip
(1) If a function $f$ is an $\eps$-maximizer and $\eps_1>\eps$, then $f$ is also an $\eps_1$-maximizer.

\smallskip
(2) The set of $\eps$-maximizers is shift-invariant (since
convolution operators commute with shifts). 

\smallskip
(3) Any subsequence of a maximizing sequence is itself 
a maximizing sequence.

\smallskip
Let us introduce some notions to be used in the proof:
a special $\eps$-maximizer, $\delta${-near-}centered function, a tight $L_p$ sequence, and a few more.%
\eref{com:def-tight}   

\begin{definition}
\label{def:eps-maximizator}
Let $f\in L_p$ be an $\eps$-maximizer. In Section~\ref{ssec:improving}
we define a nonlinear \emph{improving operator}\ $B:\,L_p\to L_p$,
such that $\|Bf\|_p=1$ and $\|k*Bf\|_{r'}\geq \|k*f\|_{r'}$, so that $Bf$ is also an $\eps$-maximizer (Lemma~\ref{lemma:norm-increase}). The so obtained
$\eps$-maximizers will be called \emph{special $\eps$-maximizers}. 
\end{definition}

The operator $B$ appears naturally in the necessary condition
of extremum, i.e.\ the equation that must be satisfied by a
maximizer if one exists
(Section\,\ref{ssec:necessary-cond-extremum}).

A crucial property of the operator $B$ that justifies the qualifier ``improving'' (as opposed to ``not-worsening'', say)
is the fact that the 
$B$-image of a weakly convergent sequence of $L_p$ functions converges in $L_p$ norm on bounded sets
(and moreover, on any sets of finite measure) in $\RR^d$ (Lemma~\ref{lemma:Maximizer-converge}).    

\begin{definition}
\label{def:delta-support}
Given a function $f\in L_p$ and a unit vector%
\footnote{We use the standard Euclidean inner product and the Euclidean norm in $\RR^d$.}
$v\in\RR^d$, the \emph{$\delta$-diameter of the function $f$ of order $p$ in the direction $v$}\
is the nonnegative number
$$
 \diam^p_{\delta,v}(f)=\inf_{b>a}\left\{b-a\,\left|\,\int_{a<(x,v)<b} |f|^p\geq \|f\|_p^p-\delta\right.\right\}. 
$$ 
In this formula we implicitly assume that $\delta<\|f\|_p^p$. If this is not the case (in particular, if $f=0$) we define $\diam^p_{\delta,v}(f)=0$.

Suppose $a$ and $b$ are such that $b-a=\diam^p_{\delta,v}(f)$ and $\int_{a<(x,v)<b} |f|^p= \|f\|_p^p-\delta$. The existence of such $a$ and $b$ is obvious.%
\eref{com:delta-support}
We say that the segment $[a,b]$ is the \emph{$\delta$-near-support of the function $f$ of order $p$ in the direction $v$}\ and denote it $\dsupp^p_{\delta,v}(f)$. 
The function $f$ is \emph{$\delta$-near-centered of order $p$ in the direction $v$}\ if $a\le 0\le b$ and
$$
\int_{a<(x,v)<0} |f|^p=\int_{0<(x,v)<b} |f|^p= \frac{\norm{f}_p^p-\delta}{2}.
$$

Let us fix, once for all, an orthonormal basis  $\{e_1,\dots,e_d\}$ in $\RR^d$  (that is, fix the coordinate axes).

We say that the function $f$ is \emph{$\delta$-near-centered of order $p$}\
if $f$ is $\delta$-near-centered of order $p$ in the direction $e_j$ for $j=1,\dots,d$.
\end{definition}

Clearly, any function can be made $\delta$-near-centered by means of a suitable shift.
However, different values of $\delta$ may require different shifts.

\smallskip
The transformation of functions corresponding to the argument shift by a vector $a$ will be denoted $T_a$; that is, $T_a f(x)=f(x-a)$.  

\begin{remark}
The above defined centering can be called a mass-centering. 
As a natural alternative, one might propose the geometric centering, whereby the $\delta$-near-support of a $\delta$-near-centered function in the given direction would be a symmetric interval $[-b,b]$. However this latter approach is not suitable for our proof
because Lemma~\ref{lemma:bounded_shifts} would be lost.  
\end{remark}


\begin{definition}
\label{def:concentration}
A sequence of functions $f_n\in L_p$ is \emph{relatively tight}\ if for any $\delta>0$
there exists $n_0=n_0(\delta)$ such that 
$$
\sup_{n\ge n_0} \sup_{\|v\|=1}\diam^p_{\delta,v}(f_n)<\infty.
$$ 
 
A sequence of functions $f_n\in L_p$ is \emph{tight}\
if for any $\delta>0$ it is \emph{$\delta$-near-finite}\ (of order $p$), which means that
there exist $n_0$ and a cube $Q$ in $\RR^d$
with edges parallel to the coordinate axes, such that for $n\geq n_0$
$$
 \int_{Q} |f_n|^p\ge \norm{f_n}_p^p-\delta.
$$
\end{definition}

As it turns out (see Lemma~\ref{lemma:concentration-to-strong-concentration}), a relatively tight sequence is  tight
provided the $\delta$-finiteness property holds for
just one sufficiently small (depending on $\sup_n\norm{f_n}_p$) positive $\delta$.
 
\medskip
We are ready to present a high-level structure of the proof of Theorem~\ref{thm:minimizer_exists}. 
 
\begin{proofmain} 
Introduce the following classes of function sequences in $L_p$ defined in terms of imposed constraints.

\begin{enumerate}

\item
The class $\Seq{Max}$ comprises all maximizing sequences (for the convolution operator $K:\,L_p\to L_{r'}$). 

\item
The class $\Seq{SMax}$ comprises all special maximizing sequences,
that is, maximizing sequences  of the form $(Bf_n)$, where $(f_n)\in\Seq{Max}$. 

\item
The class $\Seq{RTgt}$ comprises all relatively tight sequences.%

\item
The class $\Seq{Tgt}$ comprises all tight sequences.

\item
The class $\Seq{WCvg}$ comprises all weakly convergent sequences. 

\item
The class $\Seq{LCvg}$ comprises all locally convergent sequences, i.e.\ sequences converging in $L_p$ norm on any bounded measurable subset of $\RR^d$.

\item
The class $\Seq{Cvg}$ comprises all sequences converging in $L_p(\RR^d)$.

\end{enumerate}

Note that a subsequence of a sequence that belongs to
any of these classes also belongs to that class.

\smallskip
{\bf The construction.}

1. We start out with an arbitrary sequence $(f_n)\in\Seq{Max}$.

2. Applying the operator $B$ to its members we get the sequence $(Bf_n)\in\Seq{SMax}$.

3. In view of the inclusions $\Seq{Max}\subset\Seq{RTgt}$ (Corollary~\ref{cor:maximizer-concentration}) and 
$\Seq{SMax}\subset\Seq{Max}$
(Corollary~\ref{cor:B_preserves_maximization}), we have
 $(Bf_n)\in\Seq{RTgt}$.

4. Putting $\delta_0=1/4$ (any $\delta_0<1/3$ is as good), we can find vectors $a_n$ such that the shifted functions 
$g_n=T_{a_n}(Bf_n)$ are  $\delta_0$-near-centered (Lemma~\ref{lemma:bounded_shifts}).

5. Lemma~\ref{lemma:subseq-strong-concentration}\ implies $(g_n)\in\Seq{Tgt}$.

6. The operator $B$ commutes with shifts. (It is a rather trivial fact, yet it is stated as Lemma~\ref{lemma:B-translation-invariant}).
Consequently, $g_n=B(T_{a_n}f_n)$. The class $\Seq{Max}$ is shift-invariant, hence
$(g_n)\in\Seq{SMax}\cap\Seq{Tgt}$.

7. (This is the most subtle step from the logic of proof viewpoint: we ``undo'' the operator $B$ in order to
select a subsequence \emph{in the pre-image}).
The sequence $(T_{a_n}f_n)$ is bounded in $L_p$, 
hence it contains a weakly converging subsequence, which we denote $(\shift f_m)$, avoiding
multilevel subscripts.

8. Put $\shift g_m=B\shift f_m$. Then, on the one hand, $(\shift g_m)$ is a subsequence of the sequence $(g_n)$ thus inheriting
the class memberships of the latter. On the other hand,  
since $(\shift f_m)\in\Seq{Max}\cap\Seq{WCvg}$, Lemma~\ref{lemma:Maximizer-converge} implies
$(\shift g_m)\in\Seq{LCvg}$. As a result, $(\shift g_m)\in\Seq{Max}\cap\Seq{LCvg}\cap\Seq{Tgt}$. 

9. Applying Lemma~\ref{lemma:final} we conclude that $(\shift g_m)\in\Seq{Max}\cap\Seq{Cvg}$. Let $h=\lim_{m\to\infty} \shift g_m$.
Then $\norm{h}_p=1$ and $\norm{Kh}_{r'}=\lim\norm{K\shift g_m}_{r'}=\norm{K}_{p,r}$  by continuity.
The function $h$ is a maximizer.
\end{proofmain}

Theorem~\ref{thm:minimizer_exists} is proved modulo the statements referred to at the various steps of the construction.
Proofs of all those statements are given in the next two sections.
 
The proof of Lemma~\ref{lemma:concentration_property}, whose Corollary~\ref{cor:B_preserves_maximization} is used at Step~3, is the longest. We devote the whole Section~\ref{sec:non-spreading}\ to it and break the proof into short steps. 
The other results referred to in the above construction 
are proved in Section~\ref{sec:lemmas}.

\section{Estimates for \texorpdfstring{$\delta$}{delta}-diameters of near-maximizers}
\label{sec:non-spreading}

The main technical result of this Section is 
Lemma~\ref{lemma:concentration_property}, while the conceptual  
conclusion is Corollary~\ref{cor:maximizer-concentration}. 
We approach the proof of Lemma~\ref{lemma:concentration_property} through a chain of preparatory results, of which 
all but Lemma~\ref{lemma:controlled_support} are very simple.

The indicator function of a set $\Omega$ will be denoted 
$\indf{\Omega}$; if the set $\Omega$ is defined by means of a property (or predicate) $P$, then the indicator function is written as $\indf{P}$.

\begin{lemma}
\label{lemma:u}
	If $\gamma>1$, $\lambda\in(0,1/2)$ and $u\in[\lambda,1-\lambda]$, then
	$$u^\gamma+(1-u)^\gamma\le 1-\kappa,$$
	where $\kappa=\kappa(\lambda,\gamma)=2\lambda\rbr{1-2^{1-\gamma}}>0$.
\end{lemma}
\begin{proof}
	For $\gamma>1$ the function $h(u)=u^\gamma+(1-u)^\gamma$ is convex and symmetric about  $u=1/2$. We may assume that $u\in[\lambda,1/2]$. 
	By convexity, we have the chain of inequalities
	\begin{multline*}
	h(u)=h((1-2u)\cdot 0+ 2u\cdot 1/2)\le (1-2u)h(0) + 2u h(1/2)=\\
	=(1-2u)+2u\cdot 2(1/2)^\gamma=1-2u(1-2^{1-\gamma})\le 1-2\lambda(1-2^{1-\gamma}),
	\end{multline*}
	which proves the Lemma.%
	\eref{com:subadditivity_lemma}
\end{proof}

\begin{lemma}\label{lemma:h1+h2}
    Let $\Omega$ be a measure space, $\gamma>1$, and $0<\lambda<1/2$. 
    Suppose that $g\in L_1(\Omega)$ with norm $\|g\|_1=1$ is split into the sum $g=g_1+g_2$ and the summands satisfy  
    $\|g_i\|_1\geq \lambda$ ($i=1,2$)     
    and $g_1g_2=0$.
	Then 
    $$
       \|g_1\|_1^\gamma+\|g_2\|_1^\gamma\leq 1-\kappa,
    $$
    with $\kappa=\kappa(\lambda,\gamma)$, the same as in Lemma~{\rm\ref{lemma:u}}.
\end{lemma}
\begin{proof}
	Since $g_1 g_2=0$, we have $\|g_1\|_1+\|g_2\|_1=1$.
	It remains to apply Lemma~\ref{lemma:u} with $u=\|g_1\|_1$.
\end{proof}

\begin{lemma}
\label{lemma:midpoint-average}
Let $R>0$ and $f\in L_1([-R,R])$. Then for any $a<R$ there exists $t_0$, $|t_0|\leq R-a$ such that
$$
 \frac{1}{2a}\int_{|t-t_0|\leq a} |f(t)|\leq \frac{1}{R}\|f\|_1.
$$ 
 \end{lemma}

\begin{proof} Suppose $a\le R/2$; otherwise the inequality is a tautology.
The function $h(t)=\int_{|x-t|\leq a} |f(x)|$ defined for $|t|\leq R-a$ is continuous and satisfies the inequality
$$
 2(R-a)\,\min h(t)\leq \int_{|t|\leq R-a} h(t)\leq 2a\|f\|_1.
$$	
It suffices to choose $t_0$ as the point of minimum, $h(t_0)=\min h(t)$ and recall that $2(R-a)\ge R$.%
\eref{com:midpoint-average_lemma}	
\end{proof}

\begin{lemma}
\label{lemma:midpoint-average-p}
Let $f\in L_p(\RR^d)$. Given $R>a>0$, a unit vector $v\in\RR^d$ and $c\in\RR$, 
there exists $t_0\in[c-(R-a),c+(R-a)]$ such that
$$
 \frac{1}{2a}\int_{|(v,x)-t_0|\leq a} |f(x)|^p\leq \frac{1}{R} \int_{|(v,x)-c|\leq R} |f(x)|^p.
$$ 
\end{lemma}

\begin{proof}
We may assume that $v=(1,0,\dots,0)$ and $c=0$. The result follows by applying Lemma~\ref{lemma:midpoint-average} to the one-variable function
$x_1\mapsto \int |f(x)|^p\,dx_2\dots dx_d$ considered on the interval $x_1\in[-R,R]$.
\end{proof}

\begin{definition}
Let $A$ be a map from $L$ to $\tilde L$, where $L$ and $\tilde L$ are some spaces of measurable functions in $\RR^d$. Suppose $a>0$ and a unit vector $v$ in $\RR^d$ are given.
We say that the map $A$ is an \emph{$a$-expander in the direction $v$}\ if the property
$f(x)=0$ a.e.\ for $t_1<(x,v)<t_2$, where $-\infty\leq t_1<t_2\leq+\infty$, implies the property $Af(x)=0$ a.e.\ for  $t_1+a<(x,v)<t_2-a$. 
\end{definition}

The next Lemma utilizes the notions and notation introduced
in Definition~\ref{def:delta-support}.

\begin{lemma}
\label{lemma:controlled_support}
Let $A$ be a linear bounded operator from $L_p(\RR^d)$ to $L_s(\RR^d)$, where $1\le p<s<\infty$.
 Suppose $A$ is an $a$-expander in the direction $v$. 
 Let $\|f\|_p=1$ and $D=D_{\delta,v}^p(f)$.  For any $\beta>0$
 there are two possibilities: (i) either $D\leq 8\beta a$ or (ii) $D>8\beta a$ and
 \beq{op-bdd-expansion}
 \norm{Af}_s^s< 
 \norm{A}^s 
 \rbr{1-\kappa+
 \beta^{-\ga}},
 \eeq
 where $\gamma=s/p$ and $\kappa=\kappa(\delta/2,\gamma)=\delta(1-2^{1-\gamma})$, consistent with notation in Lemma~{\rm\ref{lemma:u}}.
\end{lemma}

\begin{proof} 
The cases $\delta\ge 1$ (where $D=0$) and $D\le 8\beta a$ are trivial. Therefore we assume that $\delta<1$ and $D>8\beta a$. 
Let $c$ and $R\ge 2a$ are such that
$$
  \norm{f \indf{(v,x)>c+R}}^p_{p}=\norm{f\indf{(v,x)<c-R}}^p_{p}\ge \frac{\delta}{2}.
$$
Clearly, $D\leq 2R$. 
By Lemma~\ref{lemma:midpoint-average-p} there exists $t_0\in [c-(R-2a),c+(R-2a)]$ such that
\beq{eqn:estmz0}
 \int\limits_{|(v,x)-t_0|\leq 2a} |f|^p
 \le \frac{4a}{R} \int\limits_{|(v,x)-c|\leq R} |f|^p 
 \le \frac{8a}{D} <\beta^{-1}.
\eeq
Denote (see Fig.~\ref{fig:domains_and_supports})
$$
  f_l(x)=f(x) \indf{(v,x)<t_0},\qquad 
  f_r(x)=f(x) \indf{(v,x)\ge t_0}.
$$
Then $f(x)=f_l(x)+f_r(x)$ and
$f_l(x)f_r(x)=0$, $\|f_l\|_p^p\ge \delta/2$, $\|f_r\|_p^p\ge\delta/2$.
Applying Lemma~\ref{lemma:h1+h2} to the pair of functions $g_1=|f_l|^p$ and $g_2=|f_r|^p$ we get%
 \eref{com:prevent_splitting}
$\|f_l\|_p^s+\|f_r\|_p^s\le 1-\kappa(\delta/2,\gamma)$. 
 
Introduce an yet another function,
 $$
  f_m(x) = f(x) \indf{|(v,x)-t_0|< 2a}.
 $$
 By \eqref{eqn:estmz0},
$
\norm{f_m}_p^s
<\beta^{-\ga}$.

\begin{figure}
\centerline{
	\includegraphics[width=0.6\textwidth]{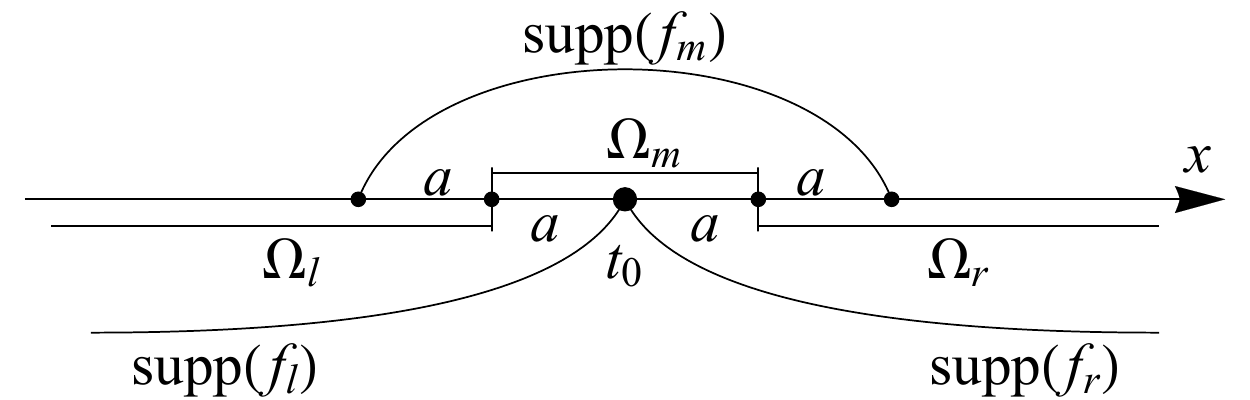}
}
\caption{Illustration of notation used in the proof of Lemma~\ref{lemma:controlled_support}}	
 \label{fig:domains_and_supports}
\end{figure}

The subsets in $\RR^d$ defined by the inequalities 
	\begin{align*}
	\Omega_l&=\fbr{x\,\left|\, (v,x)< t_0-a\right.},\\ 
	\Omega_m&=\fbr{x\,\left|\, |(v,x)-t_0|\le a\right.},\\
	\Omega_r&=\fbr{x\,\left|\,(v,x)> t_0+a\right.}
	\end{align*}
are pairwise disjoint and $\Omega_l\cup \Omega_m\cup \Omega_r=\RR^d$. We have
	\begin{align*}
	Af_{r}=0\;\;\text{при}\; x\in\Omega_l,\\ 
	Af_{l}=0\;\;\text{при}\; x\in\Omega_r,\\ 
	A(f-f_m)=0\;\;\text{при}\; x\in\Omega_m.
	\end{align*}
Therefore
\begin{align*}
	\norm{Af}_s^s&=\int_{\Omega_l} |Af_l|^s+\int_{\Omega_r} |Af_r|^s+\int_{\Omega_m} |Af_m|^s \le
	\\
	 &\le \norm{A}^s \rbr{\norm{f_l}_p^s+\norm{f_r}_p^s+\norm{f_m}_p^s}< 
	\\
	&< \norm{A}^s \rbr{ 1-\kappa+
    \beta^{-\ga}}.
\end{align*}
Q.E.D.
\end{proof}

\begin{lemma}
\label{lemma:controlled_support_optimized}
Suppose the operator $A$ satisfies the assumptions of Lemma~{\rm\ref{lemma:controlled_support}}. Suppose also that $\norm{f}_p=1$ and $\norm{Af}_s^s\geq \norm{A}^s (1-\tau)$. Then
for any $\delta>\tau(1-2^{1-\gamma})^{-1}$ 
the $\delta$-diameter 
$D=D_{\delta,v}^p(f)$
satisfies the inequality
$$
D\leq 8a(\kappa-\tau)^{-1/\gamma},
 \qquad
\kappa=\delta(1-2^{1-\gamma}).
$$ 
 \end{lemma}

\begin{proof}
Put $\beta=(\kappa-\tau)^{-1/\gamma}$ and apply Lemma~\ref{lemma:controlled_support}. Suppose the case (ii) takes place. Then
$$
 1-\tau\leq\frac{\norm{Af}_s^s}{\norm{A}^s}<
 1-\kappa+\beta^{-\ga}=1-\tau,
$$
a contradiction. Therefore the case (i) takes place and
we are done. 
\end{proof}

\begin{lemma}
\label{lemma:concentration_property}
Let $q>1$, $k\in L_{q}(\RR^d)$, and let $K:\,L_p\to L_{r'}$ be the convolution operator with kernel $k$. Put $\gamma=r'/p>1$. 
Suppose $\eps>0$ and $\delta>\eps r' (1-2^{1-\gamma})^{-1}$
are given. If $\rho>0$ 
is small enough, so that
$$
 \eps+\frac{2\rho^{1/q}}{\norm{K}_{p,r}}\leq 
 \delta \frac{1-2^{1-\gamma}}{r'}, 
$$
then for any unit vector $v\in\RR^d$ and any $\eps$-maximizer $f$ of the operator $K$ the inequality
\beq{delta-diam-estimate}
\diam^p_{\delta,v}(f)\leq c\diam^q_{\rho,v}(k),
\eeq
holds with
$$
 c=4\rbr{\delta(1-2^{1-\gamma})-r'\rbr{\eps+\frac{2\rho^{1/q}}{\norm{K}_{p,r}}}}^{-1/\gamma}.
$$
\end{lemma}

\begin{remark}
\label{rem:small_norm}
The function $\rho\mapsto D^q_{\rho,v}(k)$ is nonincreasing. Hence for fixed $\eps$ and $\delta$, given two kernels $k$ and $\tilde k$ 
with $|k|=|\tilde k|$ a.e.,  
a \emph{weaker}\ estimate (i.e.\ a smaller value of $\rho$ or larger value of $c$ in the r.h.s.\ of the inequality \eqref{delta-diam-estimate}) 
takes place for that of the two kernels with \emph{smaller}\ norm of the corresponding convolution operator. 
\end{remark}

\begin{proof}
Put $M=\norm{K}_{p,r}$ and $a=\frac{1}{2}D_{\rho,v}^q(k)$. Without loss of generality we may assume that
$\dsupp_{\rho,v}^q k=[-a,a]$. 
Let $\keps=k \indf{|(v,x)|\le a}$ and $\Keps$ is the convolution 
operator with kernel $\keps$.
We have $\norm{\keps-k}_q^q=\rho$ and, by Young's inequality, $\norm{\Keps-K}_{p,r}\leq \rho^{1/q}$.  In particular, $\norm{\Keps}_{p,r}\le M+\rho^{1/q}$.

Fix an $\eps$-maximizer $f\in L_p$ for the operator $K$.
We have $\norm{\Keps f}_{r'}\ge\norm{K f}_{r'} -\norm{\Keps-K}_{p,r}\norm{f}_p\ge M(1-\eps)-\rho^{1/q}$.
Therefore,
$$
\frac{\norm{\Keps f}_{r'}}{\norm{\Keps}_{p,r}}\ge\frac{M(1-\eps)-\rho^{1/q}}{M+\rho^{1/q}}>1-\eps-\frac{2\rho^{1/q}}{M}.
$$

The operator $\Keps$ is an $a$-expander in the direction $v$.	
Let us apply Lemma~\ref{lemma:controlled_support_optimized} with $A=\Keps$ and $s=r'$.
We have $\norm{Af}_s^s= \norm{A}^s (1-\tau)$,
where
$$
 1-\tau=\rbr{1-\eps-\frac{2\rho^{1/q}}{M}}^{r'}
 >1-r'\rbr{\eps+\frac{2\rho^{1/q}}{M}}>1-\delta(1-2^{1-\gamma})
$$
(due to the Bernoulli inequality and the inequality relating $\eps$, $\delta$ и $\rho$).
	
The estimate for $D$ provided by Lemma~\ref{lemma:controlled_support_optimized} yields \eqref{delta-diam-estimate}.
\end{proof}

\begin{corollary}
Let $q$, $k$ and $K$ be as in Lemma~{\rm\ref{lemma:concentration_property}}. Suppose that $\eps$, $\delta$, $\rho$ и $c$ are related by the equalities
\beq{diam-est-special-parameters}
 \delta=\frac{4r'}{1-2^{1-\gamma}}\eps,\quad \rho=\rbr{\norm{K}_{p,r}\eps}^q,
 \quad
 c=4(\eps r')^{-1/\gamma}.
\eeq
Then for any unit vector $v\in\RR^d$ and any $\eps$-maximizer $f$ of the operator $K$ the estimate \eqref{delta-diam-estimate} holds.
\end{corollary}

\begin{corollary}
\label{cor:maximizer-concentration}
Any maximizing sequence $(f_n)$ of $L_p$ functions for the convolution operator $K$
is relatively tight.
\end{corollary}

Indeed, let $f_n$ be an $\eps_n$-maximizer and $\eps_n\to 0$. Given a $\delta>0$ we define $\eps$ and $\rho$ by \eqref{diam-est-special-parameters} and choose $n_0$ in Definition~\ref{def:concentration} by the condition
$\eps_n\leq\eps$  for $n\geq n_0$.

\section{Lemmas for construction of a convergent maximizing sequence}
\label{sec:lemmas}

Recall that we always assume the relation $1/q+1/p+1/r=2$.

Introduce the operation $z\mapsto\pow{z}{\gamma}=\overline{z}|z|^{\gamma-1}$, where $z\in\CC$, $\gamma\in\RR$, and the bar stands for complex conjugation.
Thus, $z\pow{z}{\gamma}=|z|^{\gamma+1}$
and $|\pow{z}{\gamma}|=|z|^\gamma$.

\subsection{Auxiliary numerical inequalities}
\label{ssec:num_ineq}

\begin{lemma}
For any $u,v\in\CC$ the following inequalities hold:\\
{\rm (a)} for $0<\gamma\leq 1$,
\label{lemma:aux-ineq}
\beq{power-difference-ineq}
  \abr{\pow{u}{\gamma}-\pow{v}{\gamma}}\le C|u-v|^{\gamma},
  \quad
  C=2^{1-\gamma};
\eeq
{\rm (b)} for $\gamma>1$,
\beq{power-difference-ineq2}
  \abr{\pow{u}{\gamma}-\pow{v}{\gamma}}\le C|u-v|\,\rbr{\max(|u|,|v|)}^{\gamma-1},\quad
  C={\gamma}.
\eeq
\end{lemma}

(This Lemma will be used in the proof of Lemma~\ref{lemma:power-contunuous}.)

\begin{proof} 
(a) Put $u/v=re^{i\phi}$. 
Due to the symmetry between $u$ and $v$ we may assume that $r\le 1$. 
The inequality \eqref{power-difference-ineq} reduces to the following:
$$
 \abr{r^\gamma e^{i\phi}-1}\le C \abr{re^{i\phi}-1}^{\gamma}.
$$
Using the Cosine Theorem  
%
%
and putting $t=2r/(r^2+1)\le 1$, we can restate the required inequality in the form 
$$
 \frac{r^{2\gamma}+1}{(r^2+1)^{\gamma}}-2^{1-\gamma} t^{\gamma}\cos\phi \le C^2 \rbr{1-t\cos\phi}^{\gamma}.
$$
By concavity, $(r^{2\gamma}+1)/2\le ((r^2+1)/2)^{\gamma}$, hence
$$
 \frac{r^{2\gamma}+1}{(r^2+1)^{\gamma}}-2^{1-\gamma} t^{\gamma}\cos\phi\le
 2^{1-\gamma}(1-t^\gamma\cos\phi). 
$$
If $\cos\phi\ge 0$, then 
$$
1-t^\gamma\cos\phi\le 1-t\cos\phi\le (1-t\cos\phi)^{\gamma},
$$
and \eqref{power-difference-ineq} holds (even with a better constant), since $2^{1-\gamma}<C^2$.

If $\cos\phi<0$, then, again due to concavity, we get 
$$
1-t^\gamma\cos\phi\le 1+|t\cos\phi|^\gamma\le 
2^{1-\gamma}(1+|t\cos\phi|)^\gamma=
2^{1-\gamma}(1-t\cos\phi)^\gamma, 
$$
and the proof of the inequality \eqref{power-difference-ineq} is complete.

\smallskip
(b) Similarly, in order to prove the inequality \eqref{power-difference-ineq2} it suffices to show that for $0\leq r\leq 1$
и $\lambda=\cos\phi\in[-1,1]$
$$
 r^{2\gamma}-2r^\gamma\lambda+1\le \gamma^2(r^2-2r\lambda+1).
$$
Comparing the right-hand sides of the identities
$t^2-2t\lambda+1=(1-t)^2+2(1-\lambda)t$ with $t=r^\gamma$ and $t=r$
and using the Bernoulli inequality $1-r^\gamma\le \gamma(1-r)$,
we get the required result. 
\end{proof} 

\subsection{The improving operator}
\label{ssec:improving}

Denote $\alpha=p'/p=p'-1$, $\beta=r'/r=r'-1$, $\tilde{h}(x)=h(-x)$. 
Transposition of the convolution operator amounts to changing the original kernel into the kernel
with tilde, i.e.
$$
\rbr{k*f,g}=\int\int k(x-y)f(y)g(x)\,dy\,dx 
=\rbr{f,\tilde{k}*g}.
$$     
Clearly,
$$
 \norm{K}_{p,r}=\sup_{\|f\|_p=1,\;\|g\|_r=1}\abr{\rbr{k*f,g}}=\sup_{\|f\|_p=1,\;\|g\|_r=1}\abr{\rbr{f,k*g}}=\norm{K}_{r,p}.
$$

Let $S_p$ be the operator of radial projection onto the unit sphere in $L_p$,
$$
 S_p f=\frac{f}{\norm{f}_p}.
$$ 
Hereinafter we assume that the function acted upon by the operator $S_p$ is nonzero. 

Suppose $k\in L_q$, $f\in L_p$, so that $k*f\in L_{r'}$
and hence $\pow{(k*f)}{\beta}\in L_r$. 
Introduce the operator $B^p_r:L_p\to L_{r}$ by the formula
 $$
 B^p_r f = S_r\rbr{\pow{(\widetilde{k*f})}{\beta}}.
 $$
 (Its domain is the set $\{f\in L_p\,|\, k*f\neq 0\}$.)
Interchanging the exponents $p$ и $r$ we have
the operator
$B^r_p:L_r\to L_{p}$. Explicitly,
 $$
 B^r_p g= S_p\rbr{\pow{(\widetilde{k*g})}{\alpha}}.
 $$
\emph{The improving operator}\ $B:\,L_p\to L_p$ is the composition
 $$
 Bf= B_p^r B_r^p f.
 $$
 
 \begin{remark}
 \label{rem:symmetricB}
In the  <<symmetric>> case $r=p$ the operator $\tilde B:\,f\mapsto B^p_p \tilde{f}$,
whose square is $B$, is already a self-map of $L_p$.
As such, it can be used for the purposes of the proof
instead of the operator $B$. With this approach,
the case $\gamma\le 1$ in Lemma~\ref{lemma:power-contunuous} is not needed;
also the proof of Lemma~\ref{lemma:Maximizer-converge} becomes one-step. One property of the operator $B$ that $\tilde B$
lacks is the analog of the necessary condition of extremum $Bf=f$ (see Proposition~\ref{propos:eq-for-maximizer} in Section~\ref{ssec:necessary-cond-extremum}). One can instead propose  that a maximizer in the case $p=r$ must satisfy the equation $\tilde Bf=T_af$ with $T_a$ a shift. We do not know whether this condition is indeed necessary.       
 \end{remark}

\begin{lemma}
\label{lemma:norm-increase}
Let $f\in L_p$, $\norm{f}_p=1$, and $\norm{k*f}_{r'}>0$. Then 
\begin{equation}
\label{eqn:norm-increase}
\norm{k*B^p_r f}_{p'}\ge \|k*f\|_{r'}.
\end{equation}
\end{lemma}
\begin{proof}
Using the definition of the operator 
 $B^p_r$, we rewrite the inequality \eqref{eqn:norm-increase} to be proved in the form
$$
\norm{k*\pow{(\widetilde{k*f})}{\beta}}_{p'} \ge \norm{\pow{(\widetilde{k*f})}{\beta}}_{r}\, \norm{k*f}_{r'}=\norm{k*f}_{r'}^{r'}.
$$
(The identities $\norm{\pow{h}{\beta}}_{r}=\rbr{\int|h|^{\beta r}}^{1/r}=\norm{h}_{r'}^{\beta}$ and $\beta+1=r'$
are used.)

Since 
$\|\tilde{f}\|_p=\|f\|_p=1$, the left-hand side 
is estimated as
$$
\norm{k*\pow{(\widetilde{k*f})}{\beta}}_{p'}\ge 
\abr{\rbr{k*\pow{(\widetilde{k*f})}{\beta},\tilde{f}}}=
\abr{\rbr{\widetilde{k*f},\pow{\widetilde{k*f}}{\beta}}}
=\int|k*f|^{\beta+1}.
$$
The lemma is proved.
\end{proof}


\begin{onecorollary}
\label{cor:B_preserves_maximization} 
If $0<\eps<1$ and the function $f$ is an $\eps$-maximizer for the convolution operator $K:\,L_p\to L_{r'}$, then it belongs to the domain of $B$, and $Bf$ is also an $\eps$-maximizer for the operator $K$.
\end{onecorollary}

\begin{proof}
We have $\norm{k*f}_{r'}\ge\|K\|_{p,r}(1-\eps)>0$, hence the function $g=B^p_r f\in L_r$ is defined and 
$\norm{g}_r=1$. According to \eqref{eqn:norm-increase},
$$
\norm{k*g}_{p'}\ge  \norm{k*f}_{r'}>0.
$$
Therefore the function $h=B^r_p g=Bf$, $\norm{h}_p=1$ is defined and again, according to \eqref{eqn:norm-increase} with $p$ и $r$ swapped,
$$
\norm{k*h}_{r'}\ge \norm{k*g}_{p'},
$$
whence $\norm{K(Bf)}_{r'}\ge \norm{Kf}_{r'}$, as required.
\end{proof}

\begin{lemma}
\label{lemma:B-translation-invariant}
The operator $B$ commutes with shifts: $B(T_a f)=T_a(Bf)$ for any $a\in\RR^d$. (If one side of the formula is defined, then the other is defined too.) 
\end{lemma}

\begin{proof} 
 We have $\widetilde{T_a f}=T_{-a}\tilde{f}$, therefore, $T_{-a}(B^p_r f)=B^p_r(T_{a} f)$, and similarly for $B^r_p$.
 The claimed equality follows.
\end{proof}

\begin{lemma}
\label{lemma:power-contunuous} 
Let $s>1$ and $\gamma s>1$. Then the map $Q:\,f\mapsto \pow{f}{\gamma}$ from $L_{\gamma s}$ to  $L_s$ is continuous.
\end{lemma}

\begin{proof}
Consider two cases.

1. In the case $\gamma\le 1$ the continuity of $Q$ easily follows from the numerical inequality \eqref{power-difference-ineq}:
$$
 \norm{Qf-Qg}_s^s= \int \abr{\pow{f}{\gamma}-\pow{g}{\gamma}}^s\le C^s\int\abr{f-g}^{\gamma s}=C^s\norm{f-g}_{\gamma s}^{\gamma s}.
$$

2.  In the case $\gamma>1$ we use the numerical inequality
\eqref{power-difference-ineq2} and H\"older's inequality and find
\begin{multline*}
\norm{Qf-Qg}_{s}^s\le 
C^s\int\abr{f-g}^s\,\rbr{|f|+|g|}^{(\gamma-1)s} \le
 \\
\le C^s\rbr{\int\abr{f-g}^{\gamma s}}^{1/\gamma}\,
 \rbr{\int\rbr{|f|+|g|}^{\gamma'(\gamma-1) s}}^{1/\gamma'}.
\end{multline*}
Since $\gamma'(\gamma-1)=\gamma$ and $(|f|+|g|)^{\gamma s}\le 2^{\gamma s-1}(|f|^{\gamma s}+|g|^{\gamma s})$ (by concavity),
we get
$$
\norm{Qf-Qg}_{s}^s
\le 2^{\gamma s-1} C^s\,\norm{f-g}_{\gamma s}^s\,\rbr{\norm{f}_{\gamma s}^{\gamma s/\gamma'}+\norm{g}_{\gamma s}^{\gamma s/\gamma'}}.
$$
This concludes the proof.
\end{proof}

\begin{onecorollary}
\label{cor:Bcontinuous}
The operators $B^p_r$, $B^r_p$ and $B$ are continuous on their domains with respect to the norm topologies in the preimage and image spaces.
\end{onecorollary}

\begin{proof}
Each of these operators is a composition of continuous maps; Lemma~\ref{lemma:power-contunuous} provides the continuity in
the only place where it is not a commonly known fact. 
\end{proof}

\subsection{A compactness lemma}
\label{ssec:compact}

\begin{lemma}
\label{lemma:converge-convolutions}
Let $k\in L_q(\RR^d)$ and $\chi\in L_{r'}\cap L_{\infty} (\RR^d)$. Then the integral operator with kernel $\chi(x)k(x-y)$,
$$
  \chi K:\; f(x)\mapsto \chi(x)(k*f)(x),
$$	 
maps any weakly convergent sequence
$f_n\in L_p$ to a sequence convergent in $L_{r'}$ norm. 
\end{lemma}

\begin{proof} 
Without loss of generality we may assume that $\norm{f_n}_p\le 1$ for all $n$. Let $f_n\weakto f$ in $L_p$; then $\norm{f}_p\le 1$.

Consider first the case $k\in  L_q\cap L_{\infty}$. Since $q<p'<\infty$, we have $k\in L_{p'}$, hence the sequence $k*f_n$ converges pointwise. 
Besides, $\norm{k*f_n}_\infty\le \norm{k}_{p'}\norm{f_n}_p \le \norm{k}_{p'}$, therefore
$$
|\chi(x)\cdot(k*f_n)(x)|\le \norm{k}_{p'} |\chi(x)|.
$$
The majorant in the right-hand side lies in
$L_{r'}$. By the Dominated Convergence Theorem we conclude that $\|f_n-f\|_{r'}\to 0$. 

\smallskip
Now let us withdraw the assumption $k\in L_{\infty}$. Let $K_\lambda$ be the operator of convolution with truncated function 
$k_\lambda(x)=k(x) I_{|k(x)|\le \lambda}(x)\in L_q\cap L_{\infty}$. 
As follows from the previous, $\norm{\chi K_\lambda(f_n-f)}_{r'}\to 0$.
The proof is finished by use of the $\eps/3$ trick.
Given $\eps>0$ we find $\lambda$ such that $\norm{\chi}_\infty \norm{k-k_\lambda}_q<\eps/3$. Let $n_0$ be such that
$\norm{\chi K_\lambda(f_n-f)}_{r'}<\eps/3 $ when $n\ge n_0$. Then for $n\ge n_0$ we have
$$
\norm{\chi K(f_n-f)}_{r'}\le \norm{\chi K_\lambda(f_n-f)}_{r'} +\norm{\chi (k_\lambda-k)*f_n}_{r'} +\norm{\chi (k_\lambda-k)*f}_{r'}
<\eps.
$$
(The $2$nd and $3$rd terms in the middle are estimated by Young's inequality.) The proof is complete.%
\eref{com:compactness-lemma} 
\end{proof}

\begin{onecorollary}
\label{cor:compactness}
Let $k\in L_q(\RR^d)$ and $\Omega\subset \RR^d$ be a set of finite measure. If the sequence $(f_n)$ is weakly convergent in $L_p(\RR^d)$, then the sequence of convolution restrictions $\left.(k*f_n)\right|_{\Omega}$ strongly converges in $L_{r'}(\Omega)$. 
\end{onecorollary}

\subsection{Special maximizers and strong convergence on sets of finite measure} 
\label{ssec:weak-to-strong-convergence}

\begin{lemma}
\label{lemma:Bconverge}
Let $k\in L_q(\RR^d)$ and a weakly convergent sequence $f_n\weakto f$ in $L_p$ be given. 
Put $g_n=B^p_r f_n$ and $g=m^{-\beta} \pow{(\widetilde{k*f})}{\beta}$. 
If $\norm{k*f_n}_{r'}\to m>0$ for $n\to\infty$, then for any set $\Omega\subset \RR^d$ of finite measure 
the sequence $(g_n)$ restricted onto $\Omega$ converges in $L_r(\Omega)$ norm to $g$, 
	$$
	  \norm{g_n-g}_{L_r(\Omega)}\to 0. 
	$$ 
	Also the weak convergence $g_n \weakto g$ 
	holds in $L_r(\RR^d)$.
\end{lemma}

\begin{proof}
1. Put $h_n=k*f_n$, $h=k*f$, and  $m_n=\norm{h_n}_{r'}$,
 Then $g_n=\pow{(\tilde{h_n}/m_n)}{\beta}$. 
By Corollary~\ref{cor:compactness}, $h_n\to h$ in $L_{r'}(\Omega)$. Since $1/m_n\to 1/m$, it follows that 
$h_n/m_n\to h/m$ in $L_{r'}(\Omega)$. The tilde operation commutes with passing to the limit. 
Applying Lemma~\ref{lemma:power-contunuous} we get $\pow{g_n}{\beta}\to\pow{g}{\beta}$ в $L_{r}(\Omega)$.

\smallskip
2. Let us now prove that $(g_n-g,\psi)\to 0$ for any  $\psi\in L_{r'}$. 
Suppose $\eps>0$ is given. Fix a set $\Omega$ of fonite measure and such that $\norm{\psi}_{L_{r'}(\RR^d\setminus\Omega)}\le \eps$.
By part~1, there exists $n_0$ such that $\norm{g_n-g}_{L_r(\Omega)}\le \eps$ for $n\ge n_0$. Then for $n\ge n_0$
we have
$$
\abr{(g_n-g,\psi)}\le \eps\norm{g_n-g}_{L_{r}(\RR^d\setminus\Omega)}+\eps\norm{\psi}_{L_{r'}(\Omega)}\le \eps(1+\norm{g}_r+\norm{\psi}_{r'}). 
$$
It is clear now that $\lim_{n\to\infty} (g_n-g,\psi)=0$. 

The lemma is proved.
\end{proof}

\begin{lemma}
\label{lemma:Maximizer-converge}
Let $k\in L_q(\RR^d)$ и $(f_n)$ be a maximizing sequence of $L_p$ functions for the convolution operator $K:\,L_p\to L_{r'}$. Put $h_n=Bf_n$; according to Corollary~{\rm\ref{cor:B_preserves_maximization}}, 
$(h_n)$ is also a maximizing sequence for the operator $K$. If the sequence $(f_n)$ converges weakly in $L_p$, 
then there exists a function $h\in L_p$ such that 

 (i) $h_n\weakto h$ in $L_p$;
 
 (ii) for any set $\Omega\subset \RR^d$ of finite measure,
	$\norm{h_{n}-h}_{L_p(\Omega)}\to 0$ as $n\to\infty$.
\end{lemma}

\begin{proof}
Put $g_n=B^p_r f_n$. Lemma~\ref{lemma:Bconverge} is applicable with $m=\norm{K}_{p,r}$
and it yields weak convergence of $(g_n)$ in $L_r$.
	
Due to the equality $\norm{K}_{r,p}=\norm{K}_{p,r}$ 
and Lemma~\ref{lemma:norm-increase},  $(g_n)$ is
a maximizing sequence for the operator $K:\,L_{r}\to L_{p'}$. Applying Lemma~\ref{lemma:Bconverge} again, with replacements $B^p_r\mapsto B^r_p$,
   	 $\beta\mapsto\alpha$, $f_n\mapsto g_n$ and $g_n\mapsto h_n$, we obtain the function $h= \wlim\limits_{n\to\infty} h_n\in L_p$ that possesses all the required properties.  
   %
\end{proof}

\subsection
{Shifts, centering, and tightness}

The lemmas of this subsection are but various technical
expressions of the simple idea:
if a mass is concentrated near the origin, then
a long distance shift is incompatible with
centering. 

Let us first turn to the notions introduced in 
Definition~\ref{def:delta-support}
and prove boundedness of the set of shift vectors
that provide  $\delta$-near-centering of a given function for varying but small values of $\delta$.

\begin{lemma}
\label{lemma:bounded_shifts}
Let $1\le p<\infty$.
Fix $f\in L_p$, $\delta_0<\frac{1}{3}\norm{f}_p^p$ and a unit vector $v\in\RR^d$. Put $D=\diam^p_{\delta_0,v}(f)$.
Let $a_0\in\RR^d$ be the vector for which 
the function $T_{a_0} f$ is $\delta_0$-near-centered
in the direction $v$. 
If a function $T_{a} f$ is $\delta$-near-centered in the direction $v$ for some $\delta\le \delta_0$ and $a\in\RR^d$,
then $|(a-a_0,v)|\le D$.
\end{lemma}

\begin{proof}
Assume, without loss of generality, that $v=(1,0,\dots,0)$. Introducing the one-variable function
$$
 f_1(x_1)=\int_{\RR^{d-1}} |f(x)|^p\,dx_2\,\dots\, dx_n,
$$
we reduce the general case to the case $d=1$, $p=1$, $f\geq 0$ (where $f$ now stands for $f_1$ from the line above.)

Now $a_0$ and $a$ are scalars. Let $\norm{f}_1=m$. Due to the assumed centerings we have, first, 
$$
  \int_{a_0-D}^{a_0} f 
  \ge \frac{m -\delta_0}{2},
  \qquad
  \int_{a_0}^{a_0+D} f 
  \ge \frac{m-\delta_0}{2} .
$$
Next,  
$$
  \int_{-\infty}^{a} f 
  \ge \frac{m -\delta}{2},
  \qquad
  \int_{a}^{\infty} f 
  \ge \frac{m -\delta}{2}.
$$
Suppose that $a>a_0+D$. Then
$$
 \frac{m}{3}<\frac{m -\delta}{2}\le \int_a^\infty f
  \le \int_{a_0+D}^\infty f \le m-\int_{a_0-D}^{a_0+D} f \le 
  \delta_0<\frac{m}{3},
$$
a contradiciton. Likewise, the assumption $a<a_0-D$
leads to a contradiciton. 
We conclude that $|a-a_0|\le D$.
The lemma is proved.
\end{proof}

Further lemmas of this subsection 
pertain to the notions introduced in Definition~\ref{def:concentration}.

\begin{lemma}
\label{lemma:shifted_sequence}
Suppose the sequence of vectors $a_n\in\RR^d$ is bounded.
If the sequences $(f_n)$ and $(\shift f_n)$ in $L_p$ are related by shifts, 
$\shift f_n=T_{a_n} f_{n}$, and one of them is tight,
then the other one is tight, too.
 \end{lemma}

\begin{proof}
Let $\norm{a_n}\le R$  for all $n$.
For any coordinate cube $Q$, the shifted cube $T_{a_n} Q$ is contained in the  $n$-independent cube $Q_R$ concentric with $Q$ and with side length which exceeds that of $Q$ by $2R$. Therefore for any $\delta>0$ the sequence
$(f_n)$ is $\delta$-near-finite if and only if such the same is true about the sequence $(\shift f_n)$.
The lemma is proved.
\end {proof}

\begin{lemma}
\label{lemma:subseq-strong-concentration}
Suppose the sequence $(f_n)$ в $L_p$ ($1\le p<\infty$) is relatively tight and $\norm{f_n}_p=1$ for all $n$. 
Suppose further that all the functions $f_n$ are $\delta_0$-near-centered (of order $p$) with some $\delta_0<1/3$.
Then  the sequence $(f_n)$ is tight.   
\end{lemma}

\begin{proof}
Without loss of generality we may assume that for $\delta=\delta_0$ and all vectors $e_j$ of the fixed orthonormal basis in $R^d$ 
the condition in the last part of Definition~\ref{def:concentration} holds with $n_0=1$. Thus there is $D_0>0$ such that for any $n\ge 1$,
$$
  \diam^p_{\delta_0,e_j}(f_n) < D_0.  
$$

It suffices to verify the condition of $\delta$-near-finiteness for any given $\delta>0$.

Fix $\delta$; we may assume that $\delta<1/3$.
Let us select the vectors $a_n$ so as to obtain
$\delta$-near-centered functions $\shift f_n=T_{a_n} f_n$. 
By Lemma~\ref{lemma:bounded_shifts}, for all $n\ge 1$ and $j=1,\dots,d$ we have $|(a_n,e_j)|\le D_0$. 

By definition of a relatively tight sequence,
there exist $n_0$ and $D$ such that 
$\diam^p_{(\delta/d),e_j}(f_n)\le D$ for all $n\ge n_0$ and $j=1,\dots,d$. Then 
$$
  \int_{|(x,e_j)|>D+D_0} |f_n|^p\le \int_{|(x,e_j)|>D} |\shift f_n|^p\le \frac{\delta}{d}
$$
for $n\ge n_0$ and $j=1,\dots,d$.
 
Put $R=D+D_0$. 
The complement of the cube $Q=[-R,R]^d$ 
is the union of the sets $\{x\,|\,|(x,e_j)|>R\}$, $j=1,\dots,d$.
Therefore, $\int_Q |f_n|^p\ge \norm{f_n}_p^p-d\cdot (\delta/d)=1-\delta$.

\smallskip 
The condition of $\delta$-near-finiteness is affirmed, and the Lemma is proved.
\end{proof}


The next lemma, though not used in the proof of 
Theorem~\ref{thm:minimizer_exists}, is included
as it further clarifies the connection between the notions of relative tightness and tightness. 

\begin{lemma}
\label{lemma:concentration-to-strong-concentration}
Suppose $(f_n)$ is a relative tight sequence in $L_p$ ($1\le p<\infty$) and $\norm{f_n}_p=1$ for all $n$.
If the sequence $(f_n)$ is $\delta_0$-near-finite for some $\delta_0<1/3$, then it is tight.   
\end{lemma}

\begin{proof}
Consider the $\delta_0$-near-centered sequence $(\shift f_n)$, obtained from $(f_n)$ by means of suitable shifts,
$\shift f_n=T_{a_n} f_n$. 
Let us show that the sequence of vectors $(a_n)$ is bounded in $\RR^d$.

We may assume that the coordinate cube $Q$ in the definition of $\delta_0$-near-finiteness has the origin as its center and is described by the inequalities $|(x,e_j)|\le R$, $j=1,\dots,d$.
If $(a_n,e_j)>R$, then 
$$
 \int_{(x,e_j)>0} |\shift f_n|^p\le \int_{(x,e_j)>R} |f|^p\le \delta_0, 
$$
which contradicts the function $\shift f_n$ being $\delta_0$-near-centered. Therefore, $(a_n,e_j)\le R$. Similarly $(a_n,e_j)\ge -R$.
Thus, 
$\sup_n \|a_n\|\le R\sqrt{d}$.

Applying Lemma~\ref{lemma:subseq-strong-concentration} to the sequence $(\shift f_n)$ and then applying Lemma~\ref{lemma:shifted_sequence}  to the pair of sequences
$(f_n)$, $(\shift f_n)$, we come to the conclusion as stated.
\end{proof}

\subsection{The final lemma}
\label{ssec:final_lemma}

\begin{lemma}
\label{lemma:final}
Suppose that the sequence of functions $(f_n)$ in $L_p(\RR^d)$ possesses the following properties:

\begin{itemize}
\item[{\rm(i)}] 
normalization:
$\norm{f_n}_p=1$ при всех $n$;


\item[{\rm(ii)}] 
tightness (see Definition~{\rm\ref{def:concentration}});

\item[{\rm(iii)}]  
local convergence:
there exists a function $f\in L_p$ to which $f_n$ converges on bounded sets: $\norm{I_{\Omega}(f_{n}-f)}_p\to 0$
 for any bounded set $\Omega\subset \RR^d$.	

\end{itemize}

Then $f_n\to f$ in $L_p$. In particular, $\norm{f}_p=1$.

\end{lemma}

\begin{proof} 
Let $\eps>0$ be given.
Take a bounded set $U$ such that
$\norm{I_{\RR^d\setminus U} f}_p< \eps/3$.

Due to the assumptions (i) and (ii), there are $n_1$ and a cube $Q$ in $\RR^d$ such that 
$\norm{I_{\RR^d\setminus Q} f_n}_p< \eps/3$
for $n\ge n_1$. 

Put $\Omega=U\cup Q$.
Due to the assumption (iii), there is $n_2$ such that $\norm{I_{\Omega} \cdot (f_n-f)}_p< \eps/3$
for $n\ge n_2$. 
Clearly, for $n\ge \max(n_1,n_2)$ we have the inequalities
$$
 \norm{f_n-f}_p\le \norm{I_{\Omega}(f_n-f)}_p+\norm{I_{\RR^d\setminus\Omega} f_n}_p+\norm{I_{\RR^d\setminus\Omega} f}_p
 <\eps/3+\eps/3+\eps/3=\eps.
$$
The lemma is proved.
\end{proof}

\section{Supplementary results}
\label{sec:discussion}

\setcounter{subsection}{-1}
\subsection{A survey} 
\label{ssec:discussion-survey}

In this section we put together diverse, relatively simple results related to various aspects and details of formulation and proof of Theorem~\ref{thm:minimizer_exists}. 
Some other related, but unsolved questions 
are considered in Section\,\ref{sec:open-problems}. 

\smallskip
Subsection\,\ref{ssec:limit_cases}. {\em Limit cases}.
Theorem~\ref{thm:minimizer_exists} excludes the cases where at least one of the exponents $p$, $q$, $r$ in Young's inequality equals $1$ or $\infty$.
We analyse all such cases. A summary of the results is presented on Fig.~\ref{fig:existence_summary}.

\smallskip
Subsection\,\ref{ssec:compact-group}. {\em Convolution on compact groups}.
The analog of Theorem~\ref{thm:minimizer_exists} for compact groups is an easy result. The groups need not be commutative.

\smallskip
Subsection\,\ref{ssec:counterex-perturbation}. {\em Counterexample: a near-convolution without a maximizer}.
We give a counterexample to demonstrate that the assumptions
in Theorem~\ref{thm:minimizer_exists} cannot be relaxed
by allowing integral operators $K$ with non-translation-invariant kernels, even under the assumption that the kernel 
is pointwise dominated by the kernel $k(x-y)$ of an admissible convolution operator. Another possibility to relax the assumptions is to consider compact or even finite-dimensional perturbations of a convolution operator. In that case, we were unable to prove or disprove the existence of a maximizer; see Question~\ref{opq:compact-perturbation} in Section\,\ref{sec:open-problems}.

\smallskip
Subsection\,\ref{ssec:necessary-cond-extremum}. {\em Necessary condition of extremum}. 
First, using the standard Lagrange multipliers method,
we derive a nonlinear integral equation that must be satisfied
by a maximizer.
Then we prove an ``approximative'' version of the necessary condition of extremum: if the norm of 
the convolution $k*f$ is close to $\|K\|$ and
$\|f\|=1$, then $f$ satisfies the mentioned equation
up to a small error.%
\eref{com:discrepancy}    

\smallskip
Subsection\,\ref{ssec:any-subseq-conv}. {\em Convergence to a maximizer in the class $\Seq{Max}$ $($rather than in $\Seq{SMax})$}.
In the course of the proof of Theorem~\ref{thm:minimizer_exists}
we have established that any {\it special} (that is, lying in the image of the improving operator $B$) maximizing sequence becomes relatively compact after applying appropriate shifts. Here we show that the same is true for arbitrary maximizing sequences. 
 
This simple result is perhaps of minor significance, but we included it due to an authoritative motivation\eref{com:Lions-convergence}.

\smallskip
Subsection\,\ref{ssec:maximizer-for-limit-kernel}. {\em Kernel approximation and convergence of maximizers}. Given a sequence of convolution kernels
$k_n$ converging in $L_q$ to a kernel $k$, is it true that a maximizer for the operator $K_k$ can be obtained as a limit (in $L_p$) of maximizers for the operators $K_{k_n}$? 
Proposition \ref{propos:approximation-kernel} answers this question in the affirmative. The result 
can be of use, for example, when one has to compute
a maximizer for convolution with non-compact and, possibly, weakly singular kernel: the kernel can be approximated by bounded and finitely supported truncations.

\smallskip
Subsection\,\ref{ssec:maximizer-in-any-Lp}. {\em On boundedness and integrability of maximizers}

If one has ``a spare room of integrability'', $k\in L_{q-\eps}\cap L_{q+\eps}$ (as in Pearson's theorem), then a maximizer belongs to $L_{p_\#}\cap L_\infty$, where $p_\#<p$ does not depend on $\eps$. See also Question~\ref{opq:maximizer_in_L1Linfty} in Section\,\ref{sec:open-problems}.

\smallskip
Subsection\,\ref{ssec:norm-lower-bound}. {\em On the lower bound of convolution operators' norms}.
The estimate in Lemma~\ref{lemma:concentration_property} becomes less efficient as the norm of the operator $K$ decreases. (Cf.\ Remark~\ref{rem:small_norm}).
This fact has no adverse consequences for the proof of Theorem~\ref{thm:minimizer_exists}, 
but one should keep it in mind if the results of Section\,\ref{sec:non-spreading}
are to be used for obtaining uniform estimates
(over some family of kernels $k$).
In particular, suppose that the absolute value $|k(x)|$ of the kernel fixed; then how small can the norm $\|K_k\|_{p,r}$ be? Proposition~\ref{prop:zero_lower_bound} states that 
$\inf\|K_k\|_{p,r}=0$.%
\eref{com:convolution-lower-bound}

\subsection{Limit cases}
\label{ssec:limit_cases}

For $1\le p,q\le\infty$ the relation~\eqref{pqr}  defines the exponent $r\in[1,\infty]$
if and only if $1/p+1/q\ge 1$, equivalently, if $q\le p'$. If $q=p'$, then $r'=\infty$.
Consider the coordinate $(u,v)$-plane with $u=1/q$, $v=1/p'$. The domain corresponding to admissible pairs $(q,p)$ in Young's inequality is the triangle
formed by the lines 
 (I) $u=1$ (i.e. $q=1$), (II) $v=0$ (i.e. $p=1$) and (III) $u=v$  (i.e. $r'=\infty$).
Correspondingly we have three limit cases
and subcases corresponding to the vertices of the triangle.
The results are summarized on Fig.~\ref{fig:existence_summary}.

In the conditions considered below
we always assume that pointwise equalities and inequalities are fulfilled a.e.

\begin{figure}
\centerline{\scriptsize
\begin{picture}(240,126)
\put(0,8){\includegraphics{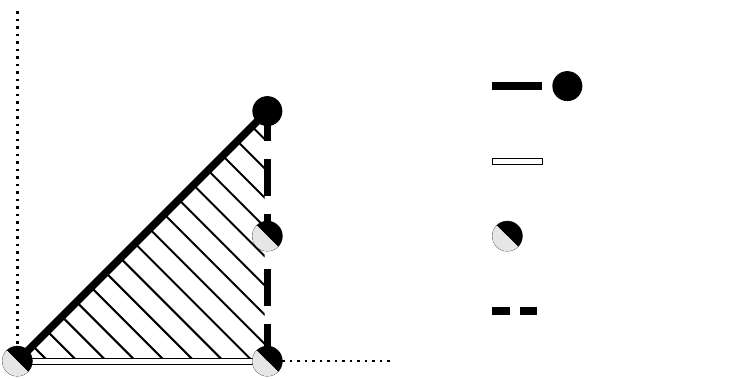}}
\put(100,17){$1/q$}
\put(8,105){$1/p'$}
\put(74,0){I(A)}
\put(84,48){I(B)}
\put(73,93){I(C)}
\put(33,2){II(A)}
\put(-5,0){II(B)}
\put(30,54){III}
\put(2.5,85.6){\line(1,0){5}}
\put(-5,83){$1$}
\put(175,90){Exists}
\put(170,68.5){Does not exist }
\put(170,47){Both possibilities}
\put(170,28){Nonexistence is possible}
\put(170,19){Existence --- ?}
\end{picture}
} 
\caption{The cases of existence/nonexistence of a  maximizer}
\label{fig:existence_summary}
\end{figure}
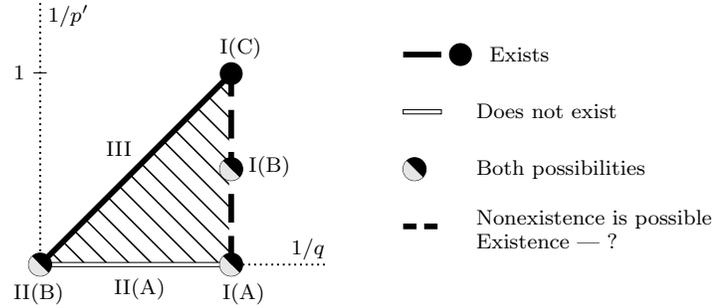

\medskip
Case I. $q=1$. 

\smallskip
Subcase I(A). $q=1$, $p=r'=1$.

\smallskip
$\mathrm{I(A)_1}$. If $k\geq 0$, then any function $f\ge 0$ with $\int f=1$  is a maximizer.
Obviuosly, the same holds true for functions of the form $k=ck_+$, where $k_+\ge 0$, $c=\const$.     

\smallskip
$\mathrm{I(A)_2}$. If $k$ is not a function with constant complex argument (in the real case --- 
a sign-changing function), then
a maximizer does not exist.  Indeed, one can choose a maximizing sequence to be a $\delta$-sequence, so
$\|K_k\|_{1,\infty}=\|k\|_1$; but the equality in the
integral Minkowski inequality
$$
 \int\abr{\int k(y)f(x-y)\,dy}\,dx\le \int |k(y)|\,dy\; \int|f(x)|\,dx
$$
is impossible (if $f\ne 0$).

\smallskip
Subcase I(B). $q=1$, $p=r'\in(1,\infty)$. 

\smallskip
$\mathrm{I(B)_1}$. Let us show that if $k\ge 0$, then there is no maximizer. 

In this case $\|K_k\|_{p,p'}=\|k\|_1$. 
Indeed, the sequence of pairs $\{f_n,g_n\}$ with   
$$
f_n(x)=n^{-1/p} I_{[0,n]}(x),
\qquad
g_n(x)=n^{-1/p'} I_{[0,n]}(x)
$$
is a maximizing sequence for the bilinear form $(k*f,g)$. 
Then $\|f_n\|_p=\|g_n\|_{p'}=1$ and
$$
 (k*f_n,g_n)=\frac{1}{n}\int_0^n \int_{x-n}^x k(t)\,dt\,dx=\int_{-n}^n k(t)\rbr{1-\frac{|t|}{n}}\,dt \to \int k(t)\,dt.
$$
A hypothetical maximizer $f$ would satisfy the equality
$$
 \norm{\int k(y) T_y f(\cdot)\,dy}_p=\|k\|_1\, \|f\|_p,
$$
which is the case of equality in the Minkowski integral inequality. This, in turn, would imply 
the existence of a function 
$\lambda(y)\ge0$ such that $T_y f(x)=\lambda(y) f(x)$ for almost all $x$, $y$.
But this is clearly impossible unless $f=0$.

\smallskip
$\mathrm{I(B)_2}$. 
Let us show that generally (for not-constant-sign functions) a maximizer can exist.
Let $p=r'=2$.  The operator $K_k$ acts in the Hilbert space $L_2$ and is unitary equivalent to multiplication by the continuous function $w=\mathcal{F}k$, where $\mathcal{F}k$ is the Fourier transform with unitary normalization.
Let $m=\|w\|_\infty$.
A maximizer exists if and only if the set $\{\xi:\, |w(\xi)|=m\}$ has positive measure.
This is possible. For example, let $w$ be a ``hat'' function: $w\in C_0^\infty$, $0\le w(\xi)\le 1$ everywhere, and $w(\xi)=1$ in some neighborhood of zero, $U$.
Then $k=\mathcal{F}^{-1}w\in L_1$ and $\norm{k*\mathcal{F}^{-1}I_U}_2=\norm{w\cdot I_U}_2=\norm{I_U}_2\norm{K_k}_{2,2}$.

\smallskip
The question as to whether  
a maximizer can exist in the case $p\neq 2$ is left open. (See Question \ref{opq:q1p_not2} in Section\,\ref{sec:open-problems}).

\smallskip
Subcase I(C). $q=1$, $p=r'=\infty$. A maximizer exists: for instance,  $f(x)=\overline{k(-x)}/{|k(-x)|}$. Indeed,
$\|f\|_\infty=1$ and
$$
 \|K_k\|_{\infty,1}\le\|k\|_1=(k*f)(0)\le \|k*f\|_\infty ,
$$ 
whence $\|k*f\|_\infty=\|K_k\|_{\infty,1}$.

\medskip
Case II. $p=1$. The operator $K_k$ acts from $L_1$ to $L_q$.
We assume that $q>1$, since the subcase $q=1$ is explored earlier, in I(A).

\smallskip
Subcase II(A). $p=1$, $1<q=r'<\infty$. 
A maximizer does not exist. Indeed, a $\delta$-sequence is a maximizing sequence:
$k*f_n\to k$ in $L_q$, so that $\|K_k\|_{1,q'}=\|k\|_q$. 
 The situation is similar to the one we have encountered in $\mathrm{I(B)_1}$, with functions $k$ and $f$ interchanged.
 A hypothetical maximizer $0\ne f\in L_1$ would realize the case of equality in Minkowski's inequality
$$
 \norm{\int f(y) T_y k(\cdot)\,dy}_q=\|f\|_1\, \|k\|_q,
$$
but this is impossible.

\smallskip
Subcase II(B). $p=1$, $q=r'=\infty$. 
Put $m=\|k\|_\infty$. One readily sees that $\|K_k\|_{1,1}=m$. If $k^{-1}(m)$ is a set of positive measure, then a maximizer trivially exists.
If $k^{-1}(m)$ is a set of measure zero, then 
both existence and non-existence of a maximizer
are possible. We give a partial criterion of
existence in the case of a nonnegative kernel
$k$.

\begin{propos}
Suppose that $k\in L_\infty$, $k\ge 0$, and $|k^{-1}(m)|=0$.
Put $U_a=\{x\,|\, k(x)\ge m-a\}$. 
In order for a maximizer of the convolution operator $K_k:\,L_1\to L_\infty$ to exist it is necessary that $|U_a|=\infty$ for all $a>0$, and sufficient that 
there are vectors $v_n$ such that the union
$$
 \hat U=\bigcap_{n=1}^\infty T_{v_n}\rbr{ U_{1/n} }
$$
has positive measure.
\end{propos} 

\begin{proof}
1)  Necessity.
Suppose $\inf_a |U_a|<\infty$. Then 
$\lim_{a\to 0}|U_ a|= |k^{-1}(m)|=0$.
Let $f\ge 0$ and $\int f=1$. We will prove that $f$ is not a maximizer. Take $\eps\in(0,1)$. 
Due to absolute continuity of the Lebesgue integral
(see e.g. \cite[v.\,1, Theorem 2.5.7]{Bogachev_2003}) there exists $\delta>0$ such that
 $\int_{\Omega} |f|<\eps$ for any set $\Omega$ of measure $|\Omega|<\delta$.
 Let  $a$ be such that $|U_a|<\delta$. Then for any $x$ we have $A(x)=\int_{U_a} f(x-y)\,dy < \eps$.
 Hence
 $$
 \ba{rcl}
  k*f(x)& = & \displaystyle
  \int_{U_a} k(y) f(x-y)\,dy + \int_{\RR^d\setminus U_a} k(y) f(x-y)\,dy \;\le
  \\[3ex]
  & \le & \displaystyle
  mA(x)+(m-a)(1-A(x)) \le m-a+a\eps.
  \ea
 $$
 Therefore $\|k*f\|_\infty\le m-a(1-\eps)<m$, as claimed.
 
\smallskip
2) Sufficiency. Let $\Omega\subset \hat U$ be a set of finite positive measure. We will show that $f=|\Omega|^{-1} \widetilde{I_\Omega}$
is a maximizer.
Indeed, we have
$$
 k*f(-v_n) =|\Omega|^{-1} \int k(y) I_\Omega(y+v_n)\,dy =|\Omega|^{-1} \int_{T_{-v_n}(\Omega)} k(y).
$$
But  $T_{-v_n}(\Omega)\subset T_{-v_n}(\hat U)\subset U_{1/n}$. Hence $k(y)\ge m-1/n$ whenever $y\in T_{-v_n}(\Omega)$
and $k*f(-v_n)\ge m-1/n$. 
Due to $L_1$-continuity of the shift operator, in some neighbourhood of the point $-v_n$ we have $k*f(x)\ge m-2/n$.
We conclude that $\|k*f\|_\infty \ge m-2/n$. Since $n$ is arbitrary, $\|k*f\|_\infty= m$. 

The Proposition is proved.
\end{proof} 


For example (in the one-dimensional case), for the kernels $k(x)=e^{-|x|}$ or $k(x)=|\sin x|$
there is no maximizer, while for the kernel $k(x)=1+\tanh x$  a maximizer exists.

\medskip
Case III. $r'=\infty$, $q=p'$. It suffices to assume that $1<q<\infty$, since the subcases $q=1$ and $q=\infty$ have been already covered
--- see.\ I(C) and II(B).

The present case is simple; a maximizer does exist. Given $k\in L_q$, put (using notation introduced in Section\,\ref{ssec:improving})
$f(x)=\pow{\tilde k(x)}{q/q'}$. Then
the case of equality in H\"older's inequality is realized:
$$
 k*f(0)=(k,\tilde f)=\int |k(x)|^{1+q/q'}=\|k\|_q^q =\|k\|_q \|f\|_{q'},
$$
and hence, too, the case of equality in Young's inequality:
$$
\|k*f\|_\infty=\|k\|_q \|f\|_{q'}, 
$$
so that $f$ is a maximizer.

It is instructive to compare this case with II(A), since the two cases deal with operators which are the transposes of each other. The relevant bilinear form in both cases is formally the same, however the conclusions about the existence of a  
maximizer are opposite.

Let $k\in L_q$. We fix the symbol $K$ to mean the operator of convolution with kernel $k$ acting from $L_{q'}$ to $L_\infty$. The transposed operator, acting from $L_1$ to $L_q$, as in II(A), will be denoted $K'$. We have
$$
 \|K\|=\|K'\|=\sup_{f,g\,: \|f\|_{q'}=\|g\|_1=1} (Kf,g)=\sup_{f,g\,: \|f\|_{q'}=\|g\|_1=1} (f,K'g).
$$
Our conclusions on (non-)existence of a maximizer can be expressed by means of the formula
$$
\|K\|=\|K'\|=\sup_{\|g\|_1=1} \max_{\|f\|_{q'}=1} (Kf,g),
$$
where $\sup$ cannot be replaced by $\max$. The underlying cause of the difference is of course the non-reflexivity of $L_\infty$. 
If we allow $g\in L_\infty^*$, then $\sup$ becomes attainable.
More precisely, by the Hahn-Banach theorem there exists $\gamma\in L_\infty^*$, $\|\gamma\|_{L_\infty^*}=1$ such that
$$
  (Kf_0,\gamma)=\|Kf_0\|_\infty=\|K\|, 
$$
where $f_0\in L_{q'}$ is a maximizer for the operator $K$. 
In order to describe the matters explicitly,   let us note that the image of the operator $K$ lies in the 
closed subspace $C_0\subset L_\infty\cap C$
of continuous functions vanishing at infinity. The space $C_0^*$ is the space of finite Borel measures.
The element $\gamma\in C_0^*$ realizing the equality
$$
  (Kf_0,\gamma)=\|Kf_0\|_\infty=\|K\|
$$
is the measure $\delta_{x_0}$, where $x_0$ is a point of maximum of $Kf_0(x)$.

\subsection{Convolution on compact groups}
\label{ssec:compact-group}

\begin{propos}
Let $G$ be a compact topological group with Haar measure $d\mu$, the spaces $L_p(G)$ defined with respect to this measure.
Let $k\in L_q(G)$.
Then the convolution operator $K_k: \, f(x)\mapsto \int_G k(xy^{-1})f(y)\,d\mu(y)$ acts boundedly from $L_p$ to $L_{r'}$,
where, as everywhere in this paper, $1/p+1/q+1/r=2$, and there exists a maximizer $F\in L_p(G)$:
$$
 \|F\|_p=1, \qquad \|K_k F\|_{r'}=\|K\|_{p,r}.
$$
\end{propos}


\begin{proof}
Boundedness of the operator $K_k$ (Young's inequality) is a well-known fact. (Sufficient assumption is that the group $G$ is locally compact and unimodular,
see e.g. \cite[(20.18), (20.19)]{Hewitt-Ross1}.) 
Now, take any maximizing sequence $(f_n)$ and select a weakly convergent subsequence.
The improving operator $B$ maps it to a strongly convergent one; a proof of the required analog of Lemma~\ref{lemma:converge-convolutions} is even easier here: we do not need a ``truncation in the horizontal direction'' to obtain a compaclty supported function. The limit is a maximizer.
\end{proof}

\subsection{Counterexample: a near-convolution without a maximizer}
\label{ssec:counterex-perturbation}

It is natural to ask about possible relaxation
of conditions of Theorem~\ref{thm:minimizer_exists} and to try to exhibit sufficient conditions that the kernel $K(x,y)$ of the integral operator $\tilde K:\;f(x)\mapsto \int K(x,y)f(y)\,dy$ should satisfy, not necessarily being translation-invariant, in order to guarantee the existence of a maximizer. As the example below demonstrates, conditions of such a sort,
if possible at all, cannot be stated in terms of integral and pointwise inequalities only:  
here, there is no maximizer, although we have a pointwise majorization $0<K(x,y)<k(x-y)$ with $k\in L_q$.

\begin{propos}
Let $p,q,r$ be as in Theorem~\ref{thm:minimizer_exists} and $d=1$.
Let $k\in L_q(\RR^1)$ and $k(x)>0$ everywhere. 
Consider the integral operator $\tilde K$ with kernel
$\tilde K(x,y)=k(x-y) u(y)$, 
where $u(y)$ is monotone increasing and $\lim_{y\to-\infty} u(y)=0$, $\lim_{y\to+\infty} u(y)=1$.
The operator $\tilde K:\,L_p\to L_{r'}$ is continuous. A maximizer for the operator $\tilde K$ does not exist.
\end{propos}

\begin{proof}
One readily sees that $\|\tilde K\|_{p,r}=\|K_k\|_{p,r}$. If $f\in L_p$ ($\|f\|_p=1$) is a maximizer for the operator $\tilde K$, then the function $f(x)u(x)$ must be a  maximizer for the  convolution operator $K_k$. But it is clear that $\|f u\|_{p}<1$, a contradiction.
\end{proof}

\subsection{Necessary condition of extremum}
\label{ssec:necessary-cond-extremum}

The notation from Section~\ref{ssec:improving} will be used. 
The next Proposition does not refer to the existence of maximizer result, so we allow the case $q=1$. 

\begin{propos}
\label{propos:eq-for-maximizer}
Suppose that $1\le q<\infty$ and $1<p,r<\infty$.
(The relation \eqref{pqr} is assumed as always.)	
A maximizer of the convolution operator $K$, if it exists, satisfies the equation $f=Bf$.
\end{propos}

\begin{proof}
We have the optimization problem (with given function $k$ and unknown $f$ and $g$):  
$$
 \Re\iint k(x-y) f(x) g(y)\,dx\,dy \;\to\;\max
$$
under the constraints
$$
  \int|f(x)|^p\,dx=1,\qquad \int|g(y)|^r\,dy=1.
$$
Let us use the Lagrange multipliers method to derive the system of equations to be satisfied by the extremal pair of functions $(f,g)$.
The relevant Lagrange functional can be taken in the form
$$
\mathcal{L}(f,g)=\Re\iint \tilde{k}(x+y) f(x) \tilde{g}(y)\,dx\,dy -\lambda\int|f(x)|^p\,dx-\mu\int|g(y)|^r\,dy.
$$
Computing the partial variation with respect to $f=f_1+if_2$, we get
$$
\delta\mathcal{L}=\int\rbr{G_1-\lambda p|f|^{p-2} f_1}\cdot\delta f_1-\int\rbr{G_2+\lambda p|f|^{p-2} f_2}\cdot\delta f_2,
$$
where $G=G_1+iG_2=\tilde{k}*g$.
Therefore the pair $(f,g)$ that yields an extremum of the functional $\mathcal{L}$ must satisfy the equation
$$
 \bar{G}-\lambda p|f|^{p-2} f=0.
$$
Similarly, equating the partial variation $\delta\mathcal{L}/\delta g$ to zero,  we come to the equation
$$
 \bar F-\mu r|g|^{r-2} \tilde{g}=0,
$$
where $F=\tilde{k}*\tilde{f}=\widetilde{k*f}$.

Taking into account the normalization of $f$ and $g$ and the identities $r-1=r/r'$, $p-1=p/p'$, the obtained system of equations can be written as
$$
 \tilde{g}=B^p_r f, \qquad f=B^r_p\tilde g.
$$
Elimination of $\tilde{g}$ results in the equation $Bf=f$.
\end{proof}


In Subsection~\ref{ssec:any-subseq-conv} we will need an approximative version of the necessary condition of extremum. 

\begin{propos}
\label{propos:ineq-for-maximizer}
For any $\eps>0$ there exists $\delta>0$ (depending on $q,p$ and the convolution kernel $k$)
such that if $\|f\|_p=1$ and $\|k*f\|_{r'}> \|K\|_{p,r}(1-\delta)$, then $\|Bf-f\|_p<\eps$. 
\end{propos}

\begin{proof}
We make use of the approximative version of H\"older's inequality due to H.~Hanche-Olsen \cite[Lemma~2]{Hanche-Olsen_06}:%
\eref{com:rev-Holder}	

\smallskip 
{\it For any $\eps>0$ there exists $\eta>0$ such that if $\|F\|_p=\|G\|_{p'}=1$ and $\Re (F,G)>1-\eta$, then $\|F-\pow{G}{\alpha}\|_{p}<\eps$.%
}

\smallskip
Consider an improvement of the estimate in the proof of Lemma~\ref{lemma:norm-increase}. Using the notation of Section~\ref{ssec:improving} 
(in particular, recall: $S_{p'}$ is the radial projection onto the unit sphere in $L_{p'}$, and $\beta=r'/r$), put 
$G=S_{p'}\rbr{k*\pow{(\widetilde{k*f})}{\beta}}$ и $F=\tilde{f}$.
Note that $G=\pow{(\widetilde{Bf})}{1/\alpha}$,  hence $\|F-\pow{G}{\alpha}\|_p=\|f-Bf\|_p$.

Further, denote $g=\pow{(\widetilde{k*f})}{\beta}$ and $M=\|K\|_{p,r}=\|K\|_{r,p}$.

The calculation in the proof of Lemma~\ref{lemma:norm-increase} implies that $\|k*f\|_{r'}^{r'}=(k*g,\tilde{f})=(G,F)\|k*g\|_{p'}$.
In particular, $(F,G)>0$.

We have 
$\|k*g\|_{p'}\le M\|g\|_{r}=M\|k*f\|_{r'}^{r'/r} \le  M^{\beta+1}\|f\|_p^\beta=M^{r'}$.
Assuming that $\|k*f\|_{r'}>M(1-\delta)$, we get
$$
M^{r'}(1-\delta)^{r'}<\|k*f\|_{r'}^{r'}\le  (F,G) M^{r'}.
$$
Therefore, $(F,G)>(1-\delta)^{r'}$.

Let $\eps>0$ be given. Find the corresponding $\eta$ as in Hanche-Olsen's lemma. Define $\delta$ by the equation 
$1-\eta=(1-\delta)^{r'}$. According to the above, we have the inequality $\|f-Bf\|_p<\eps$, as required.   
\end{proof}

\subsection{Convergence to a maximizer in the class \texorpdfstring{$\Seq{Max}$}{Max} (rather than in \texorpdfstring{$\Seq{SMax}$}{SMax})}
\label{ssec:any-subseq-conv}

\begin{propos}
\label{prop:any-subseq-conv}
Let $(f_n)$ be a maximizing sequence for the convolution operator $K:\, L_p\to L_{r'}$ with kernel $k\in L_q(\RR^n)$.
There exists a subsequence $(f_{n_k})$ and shift vectors $a_k$ such that the sequence $T_{a_k} f_{n_k}$
converges in $L_p$ as $k\to\infty$ (its limit automatically being a maximizer for the operator $K$). 
\end{propos}

\begin{proof}
In the proof of Theorem~\ref{thm:minimizer_exists} (see~\S\,\ref{sec:proof-general}) we found that the sequence
$(Bf_n)$ has a subsequence convergent after appropriate shifts. We may assume that the sequence
$(Bf_n)$ itself is convergent.  Proposition~\ref{propos:ineq-for-maximizer} implies that 
 $\norm{f_n-Bf_n}_{p}\to 0$ (since $\|Kf_n\|_{r'}\to \|K\|_{p,r}$.) Therefore $\lim f_n=\lim Bf_n$ does exist.
 \end{proof}

\subsection{Kernel approximation and convergence of maximizers}
\label{ssec:maximizer-for-limit-kernel}

\begin{propos}
\label{propos:approximation-kernel}
Suppose a sequence of function $k_n\in L_q$ converges (strongly) to a nonzero $k\in L_q$. Then there exists a sequence of  
maximizers $f_n\in L_p$ for the convolution operators $K_n=K_{k_n}:\,L_p\to L_{r'}$ that converges strongly to a function $f\in L_p$.
The function $f$ is a maximizer for the convolution operator $K=K_k$.
\end{propos}

\begin{proof}
An arbitrary sequence $(f_n)$ of maximizers for the operators $K_n$ is obviously a maximizing sequencefor the operator $K$.
Applying Proposition~\ref{prop:any-subseq-conv}, we obtain the claim as stated.
\end{proof}

\subsection{On boundedness and integrability of maximizers}
\label{ssec:maximizer-in-any-Lp}

\begin{propos}
\label{propos:maximizer_in_anyLp}
Let $k\in L_q$ and $f$ be a maximizer for the convolution operator $K_k$ from $L_p$ to $L_{r'}$. (We assume that neither of $p$, $q$ and $r$ is $0$ or
$\infty$.) 

{\rm (a)} If $k\in L_{q+\eps}$ for some $\eps>0$, then $f\in L_p\cap L_\infty$.

{\rm (b)} If $k\in L_{q-\eps}$ for some $\eps>0$, then $f\in L_p\cap L_{p_\#}$, where 
$$
p_\#=\left\{\ba{ll}
1, &\; r'/p'\leq q,
\\
p(1+p'/r')^{-1}, &\; r'/p'>q
\ea
\right.
$$  
($1<p_\#<p$ for $r'/p'>q$).
\end{propos}

\begin{proof}

The only information we need is that $f$ satisfies the equation $Bf=f$ (see Subsection~\ref{ssec:necessary-cond-extremum}).
Put $g=B^p_r f$ (in notation of Subsection~\ref{ssec:improving}).

\smallskip
(a) Suppose that $f\in L_s$ for $1/p-\mu\leq 1/s\leq 1/p$
with some $\mu\in[0,1/p]$. 
The identity
$$
\frac{1}{s}+\frac{1}{q}-1=\frac{1}{r'}+\left(\frac{1}{s}-\frac{1}{p}\right)
$$
shows that $k*f\in L_u$ if $\max(0,\,1/r'-\mu)\leq 1/u\leq 1/r'$. Therefore,
$g\in L_t$ if $\max(0,\,1/r-\nu)\leq 1/t\leq 1/r$,
where $\nu=(r'/r)\mu$.

Note that $\mu=1/p$ implies $\nu=(r'/p)/r>1/r$, so that $f\in L_\infty$ implies $g\in L_\infty$.

Since $f=Bf=B^r_p g$, we have similarly: if  $g\in L_t$ for $1/r-\nu \leq 1/t\leq 1/r$, $\nu\in[0,1/r]$, then
$f\in L_s$ for $\max(0,\, 1/p-\kappa)\leq 1/s\leq 1/p$, where $\kappa=(p'/p)\nu$. Also $g\in L_\infty$ implies
$f\in L_\infty$. 

\smallskip
Combining the above said, we conclude: if  $f\in L_p\cap L_P$, $P>p$,
and $\mu=1/p-1/P$, then $f\in L_p\cap L_{\tilde P}$, where
either $\tilde P=\infty$, or $(r'/r)p'\mu<1$ and
$1/p-1/\tilde P=M\mu$, where $M=(r'/r)(p'/p)$.    
Since $r'/p>1$ and $p'/r>1$, we have $M>1$.
Iterating, we get $f\in L_\infty$ in a finite number of steps.

\smallskip
The conclusion $f\in L_\infty$ is obtained under the assumpiton that
$f$ lies in $L_P$ with some $P>p$, and in the derivation we used just the inclusion $k\in L_q$. Let us now make use of the condition
$k\in L_{q+\eps}$, assuming only that $f\in L_p$. Put $\delta=1/q-1/(q+\eps)$.
Interchanging the roles of $f$ и $k$ at the first half-step of the iteration
(where we estimate the exponent of the space containing $k*f$),
we conclude that $f=Bf\in L_P$, where $1/P=\max(0,\,1/p-M\delta)$.
If $P\neq\infty$, we apply the above described iteration with initial value of parameter $\mu=M\delta$.

\smallskip
(b) Repeating the argument of part (a), we obtain: if $f\in L_s$
for $1/p\leq 1/s\leq 1/p+\mu$, then $g\in L_t$ for
$1/r\leq 1/t\leq \min(1,\,1/r+\nu)$,  $\nu=(r'/r)\mu$. Symmetrically,
 if $g\in L_t$
 for $1/r\leq 1/t\leq 1/t+\nu$, then $f\in L_s$ for
 $1/p\leq 1/s\leq \min(1,\,1/p+\kappa)$,  $\kappa=(p'/p)\nu$.

The essential difference with part (a) is that the conditions 
$f\in L_1$ are $g\in L_1$ no longer equivalent. For instance, $g\in L_1$ 
means that $\nu=1/r'$.  The value 
$1/p+\kappa=1/p+p'/(pr')$
can happen to be less than $1$.

With this remark in mind, we parallel the proof
of part (a). The condition $k\in L_{q-\eps}$ implies $f\in L_{P}$
with some $P<p$. Putting $\nu=1/P-1/p$ and $M=(p'r')/(pr)>1$,
we obtain at an iteration step: either (i) $g\in L_1$, or 
(ii) $f\in L_{1}$, or (iii) $f\in L_{\tilde P}$, 
where $1/\tilde P=1/p+M\mu<1$.
In the case (iii) we continue to iterate. Eventually, in a finite number of steps one of the cases (i) or (ii) occurs. 

The exponent $p_\#$ in the terminal case (i) is determined above: 
$1/p_\#=\min(1,\;1/p+p'/(pr'))$. 
The calculation
$$
1-\left(\frac{1}{p}+\frac{p'}{pr'}\right)=
  \frac{1}{p'}+\frac{1}{r'}-\frac{p'}{r'}=\frac{1}{q}-\frac{p'}{r'}
$$
shows that the condition $p_\#>1$ is equivalent to the inequality $r'/p'>q$.
\end{proof}

\begin{remark}
The asymmetry of the result ($f\in L_\infty$ being a ``more common'' property than $f\in L_1$) 
is ultimately due to the fact that convolution 
inherits best local properties of the two its operands, but worst global properties.
\end{remark}

\subsection{On the lower bound of convolution operators' norms}
\label{ssec:norm-lower-bound}

\begin{propos}
\label{prop:zero_lower_bound}
Let  $q\ge 1$, $1<p,r<\infty$ and $1/p+1/q+1/r=2$.
Let $k\in L_q(\RR^d)$ be a nonnegative function.  The operator of convolution with complex-valued kernel $k(x)e^{i\phi(x)}$ acting from $L_p$
to $L_{r'}$ can have arbitrarily small norm. Specifically,
$\|K_{k(x)\exp(i\lambda\|x\|^2)}\|_{p,r}\to 0$ as $\lambda\to\infty$.
\end{propos}

\begin{proof}
It is easy to see that the set of functions $k\in L_q$ for which the statement if true is closed in $L_q$. Therefore without loss of generality we may assume that
$k\in L_1\cap L_\infty$. 

Denote $K_\lambda$ the operator of convolution with function $k_\lambda(x)=k(x) \exp(i\lambda\|x\|^2)$. 

In the case $q=1$, $p=r=2$ the validity of the claim of the Proposition is established below, in Lemma~\ref{lemma:zero_lower_bound_l2}. 
The general case follows from this particular one by an interpolation argument as follows.

\smallskip
Suppose $p\ge r$ (otherwise one considers the transposed operator). 
On the coordinate plane, let us pass a line through
the points
$A=(1/2, 1/2)$ и $B=(1/p,1/r')$. 
Let $C=(\xi,0)$ be the point where it meets the horizontal axis. Due to the inequalities $1/r'<1/p\le 1-1/r'$
we have $0<\xi\le 1$. Put $\xi=1/s$, $s\ge 1$. The fact that $B\in[AC)$ can be written as
$$
 \frac{1}{p}=\frac{1-\theta}{s}+\frac{\theta}{2},
 \qquad
 \frac{1}{r'}=\frac{1-\theta}{\infty}+\frac{\theta}{2},
$$
where $0<\theta\le 1$.

Given $\eps>0$,  
Lemma~\ref{lemma:zero_lower_bound_l2} tells us that 
$\|K_\lambda\|_{2,2}\le\eps$ for a large enough $\lambda$. 
On the other hand, due to the assumption we made at the beginning of the proof, we have $k\in L_{s'}$.
By H\"older's inequality,
$$
 \|K_\lambda f\|_\infty\le \|f\|_s \|k\|_{s'}.
$$ 
Applying now the Riesz-Thorin theorem, we conclude that
$$
 \|K_\lambda\|_{p,r}\le  \|k\|_{s'}^{1-\theta}  \eps^{\theta}.
$$
The proposition is proved.
\end{proof}

\begin{lemma}
\label{lemma:zero_lower_bound_l2}
Let $k\in L_1(\RR^d)$.  Denote $k_\lambda(x)=k(x)\exp(i\lambda \|x\|^2)$ and 
$$
\hat k_\lambda(\xi)=\int k_\lambda(x) e^{-i(x,\xi)}\,dx,
$$
the Fourier transform of $k_\lambda$. Then
$\|\hat k_\lambda\|_\infty\to 0$ as $|\lambda|\to \infty$.

Consequently, the norm of the convolution with $k_\lambda$ as an operator in $L_2(\RR^d)$ tends to $0$ as $|\lambda|\to\infty$.
\end{lemma}

\begin{proof}
By a density argument, it suffices to prove the Lemma under the assumption $k(x)\in C_0^\infty(\RR)$.

For $\Re z> 0$ we have the Plancherel identity 
$$
\int \phi(x)e^{-z\|x\|^2}\,dx=Cz^{-d/2}\int \hat\phi(\xi) e^{-\|\xi\|^2/(4z)}\,d\xi.
$$ 
Both sides are defined and continuous in the region $\Re z\geq 0$, $z\neq 0$. 
Therefore the equality extends to the boundary 
$z=-i\lambda$, $\lambda\in\RR\setminus\{0\}$. Thus,
$$
\left|\int \phi(x) e^{i\lambda\|x\|^2}\right|\leq C|\lambda|^{-d/2}\|\hat\phi\|_1.
$$
Putting $\phi(x)=k(x)e^{-i(x,\eta)}$, we obtain$|\hat k_\lambda(\eta)|$ in the left-hand side of the latter inequality, while
$\hat\phi(\xi)=\hat k(\xi+\eta)$, so that $\|\hat\phi\|_1=\|\hat k\|_1$. 
The esimate
$$
 \|\hat k_\lambda\|_\infty \leq C|\lambda|^{-d/2}\|\hat k\|_1
$$
follows and the proof is complete.
\end{proof}

\section{Best constants in the Hausdorff-Young inequality for the Laplace transform on \texorpdfstring{$(0,+\infty)$}{Rplus}}
\label{sec:Laplace-norm}

Denote by $\Lt$ the Laplace transform on $\RR_+$,
$$
 f\;\mapsto\; \Lt f(x)=\int_0^\infty e^{-xt} f(t)\,dt,
$$
and by $\Ft$ the Fourier transform on $\RR$, 
$$
 f\;\mapsto\; \Ft f(x)=\int_{-\infty}^\infty e^{-ixt} f(t)\,dt.
$$
 For  $1\le p\le 2$, the Hausdorff-Young (HY) inequalities 
$$
 \|\Ft f\|_{p'}\le C^{\Ft}_p \|f\|_p
$$
and
$$
 \|\Lt f\|_{p'}\le C^{\Lt}_p \|f\|_p
$$
hold.
They are first established under the assumption $f\in L_1\cap L_{\infty}$, when the integral definitions of $\Lt f$ and $\Ft f$ have direct meaning, and then
they are used to extend $\Lt$ and $\Ft$ by continuity to
the operators acting from $L_p$ to $L_{p'}$.   

The exponent $p'$ in the left-hand sides of the HY inequalities cannot be replaced by any other number. 
This follows from ``dimensional analysis'', that is, 
changing $f(t)$ into the function $f_a(t)=a^{1/p} f(at)$ with the same $L_p$ norm, where $a>0$ is an arbitrary scaling parameter.
It is also known that inequalities of this type do not hold when $p>2$. 
In the case of Fourier transform, an explicit argument to that effect can be found, e.g., in Titchmarsch's monograph \cite[\S~4.11]{Titchmarsch1946}.

The optimal values of the constants $C^{\Ft}_p$, that is, the operator norms $\|\Ft\|_{p\to p'}$, have been found by W.\,Beckner \cite{Beckner1975b}
(and earlier by K.I.\ Babenko \cite{Babenko1961} in the case $p'/2\in\ZZ$):
\beq{Beckner-FT}
 \|\Ft\|_{p\to p'}=(2\pi)^{1/p'} A_p,
 \eeq
where the constant $A_p$ is defined in \eqref{beckC}.

Analytical expressions for the optimal values of the constants $C^{\Lt}_p$, that is, the operator norms $N(p)=\|\Lt\|_{p\to p'}$, are unknown. 
The problem of determining $N(p)$ is equivalent to the problem of determining the norm of the convolution operator with kernel $h_p(\cdot)$,
see \eqref{LT-conv-kernel} below, acting from $L_p(\RR)$ to $L_{p'}(\RR)$.

In Figure~\ref{fig:LT-norm} and in Table~\ref{tab:LT-norm} we  present the numerical values of $N(p)$. 
In order to mark the distiction between the true value of $N(p)$ and the computed approximation to it, we designate the latter as $\numN(p)$. The numerical method used is briefly outlined at the end of this Section. 

\smallskip
The minimum of the norm occurs at $p\approx 1.1307$, 
\beq{minLaplaceNorm}
\min N(p)\approx  0.881970846.
\eeq

\begin{figure}
\centerline{
\begin{picture}(330,200)
\put(0,-10){
\includegraphics[width=0.9\textwidth]{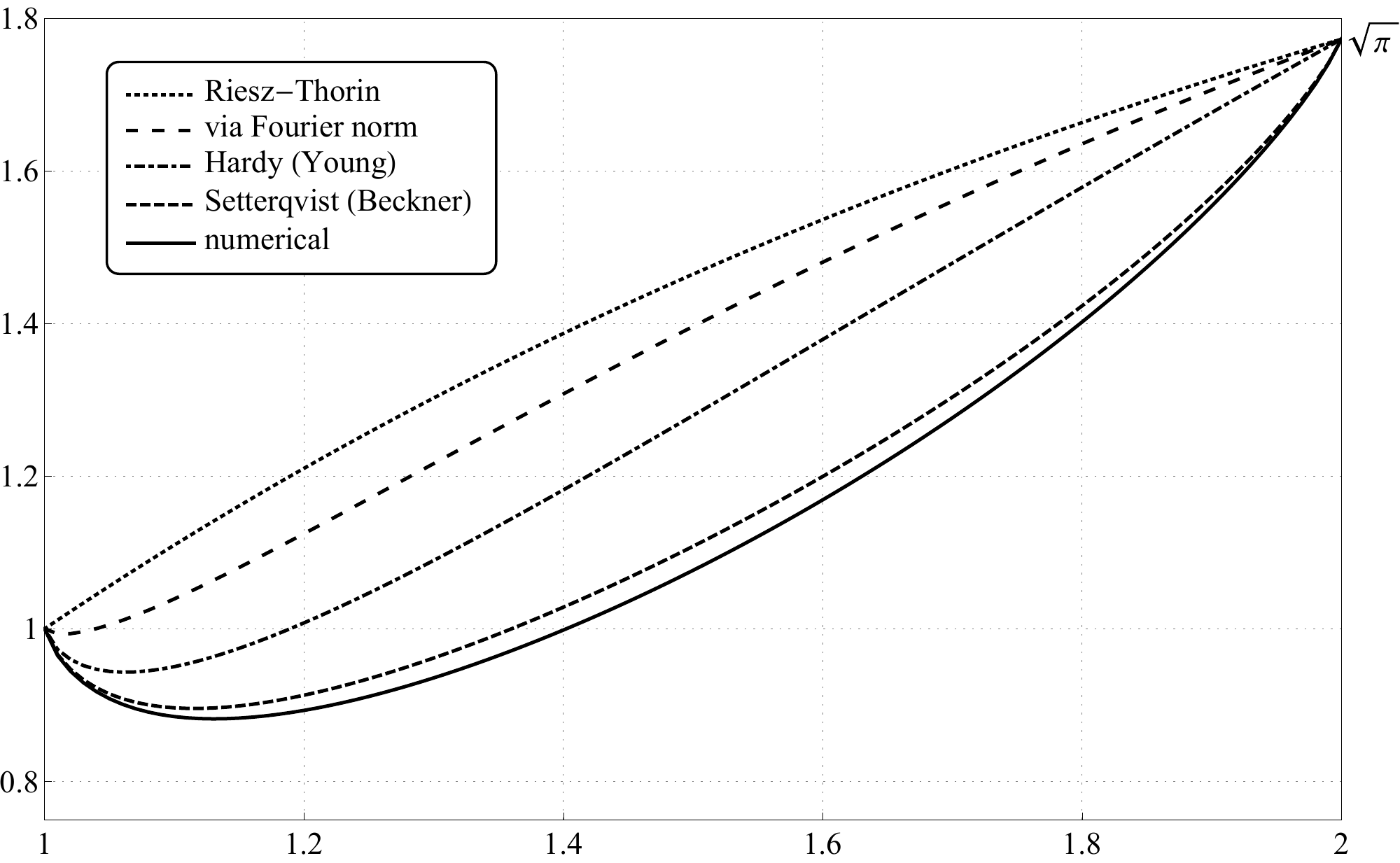} 
}
%
\end{picture}
} 
\caption{The norms $\numN(p)\approx\|\Lt\|_{p\to p'}$ found numerically in comparison with the analytical estimates \eqref{RTLp}, \eqref{FLp}, \eqref{HardyLp} и \eqref{EricLp}}
\label{fig:LT-norm}
\end{figure}


\noindent
{
 \def\arraystretch{1.2}%
\begin{table}
	\vspace{1em}
	\centerline{
		\setlength\tabcolsep{0.5\tabcolsep}
		\scriptsize
		\begin{tabular}{c|c|c|c|c|c|c|c|c|c|c}
			$p$ & $1.05$ & $1.1$ & $1.2$ & $1.3$ & $1.4$ & $1.5$ & $1.6$ & $1.7$ & $1.8$ & $1.9$  \\ 
			\hline
			$\numN(p)$ & $0.90835$ & $0.88495$ & $0.89306$ & $0.93562$ & $0.99833$ & $1.07652$ & $1.16890$ & $1.27631$ & $1.40193$ & $1.55390$ \\ 
			\hline
			$C_{SB}(p)$ & $0.91459$ & $0.89640$ & $0.91296$ & $0.96169$ & $1.02830$ & $1.10803$ & $1.19953$ & $1.30354$ & $1.42310$ & $1.56616$ \\ 
		\end{tabular}
	}
	\caption{The norms $\numN(p)\approx\|\Lt\|_{p\to p'}$ found numerically in comparison with Setterqvist's estimate \eqref{EricLp}}
	\label{tab:LT-norm}
	\vspace{1em}
\end{table}
}


\vspace{-2\medskipamount}
The curves in Fig.~\ref{fig:LT-norm} present the numerically evaluated norms $\|\Lt\|_{p,p'}$ and several analytical estimates for the norms, which we describe below.

\smallskip
1. The simplest estimate is obtained by interpolation. The equality $N(1)=1$ is immediate and the equality $N(2)=\sqrt{\pi}$ is readily obtained as the supremum of the spectrum of the self-adjoint operator $\Lt$ in $L^2(\Rp)$. The Riesz-Thorin interpolation theorem yields the estimate%
\footnote{The constants are subscripted in accordance with: RT=Riesz-Thorin, F=via Fourier norm, H=Hardy, S=Setterqvist.} 
\beq{RTLp}
 N(p)\le C_{RT}(p)= \pi^{1/p'}, \quad 1\le p\le 2.
\eeq

\smallskip
2. One can show that
\beq{FtLt}
N(p)\le 2^{-1/p'} \|\Ft\|_{p\to p'}.
\eeq
Using the Hausdorff-Young estimate $\|\Ft\|_{p\to p'}\le (2\pi)^{1/p'}$ in the right-hand side we come again to the estimate \eqref{RTLp}, but one can instead use Beckner's sharp constant. As a result, one gets a better estimate,
\beq{FLp}
 N(p)\le C_{F}=\pi^{1/p'} A_p.
\eeq
 
Let us comment on the inequality \eqref{FtLt}.
Consider the family of operators $T_z:\; L_p(\RR_+)\to L_{p'}(\RR_+)$, depending on complex parameter $z$,
$$
 T_z:\; f\mapsto T_z f(x)=\int_0^\infty e^{-\Phi(z) xy}\,f(y)\,dy,
 \qquad \Phi(z)=e^{iz\frac{\pi}2}.
$$
The analytic operator-valued function $z\mapsto T_z$ is defined in the strip $|\Re z|\le 1$ and its values at $z=\pm 1$ are
the composition of the Fourier transform with restrictions onto the negative, resp., positive real half-line. The value $T_z$ at $z=0$ is but the Laplace transform. The inequality in question follows by applying Stein's interpolation theorem \cite[Ch.~5, Theorem~4.2]{SteinWeiss_1974}.
Using this approach, the second author and A.E.~Merzon have obtained a variant of the HY inequality for the Laplace transform with variable $y$ on a ray in the complex plane (unpublished). 

\smallskip
3. The substitutions
$$
x=e^y, \qquad t=e^s, \qquad F(y)=f(e^y) e^{y/p},
\qquad 
G(s)=(\Lt f)(e^s)\,e^{s/p'},
$$ 
reduce the Laplace transform to the convolution operator 
$$
 F(y)\;\mapsto\; G(s)=\int_{\RR} h_p(y+s)\,F(y)\,dy=(h_p*\tilde F) (s), 
$$
where
\beq{LT-conv-kernel}
 h_p(y)=e^{y/p'-e^{y}}.
\eeq
It is  easy to see that the $L_p$-norms of the functions $f$ (defined on $\Rp$) and $F$ (defined on $\RR$) coincide;
the same is true for the $L_{p'}$-norms of the functions $\Lt f$ and $G$.
Therefore, $N(p)$ is the norm of the convolution operator with kernel $h_p(y)$ acting from $L_p(\RR)$ to $L_{p'}(\RR)$.  
Since $\|h_p\|_{q}^{q}=\sqrt{2\pi/p'}$ (here $q=p'/2$), the Young inequality yields the estimate 
\beq{HardyLp}
N(p)\le C_{H}(p)=\rbr{\frac{2\pi}{p'}}^{1/p'}.
\eeq 
G.H.\,Hardy \cite{Hardy1933} was the first to
derive this estimate in 1933, using the method just outlined. 

\smallskip
4. Combining Hardy's reduction with case $d=1$ of Beckner's sharp form \eqref{beckC} of Young's inequality, Setterqvist \cite[Theorem~2.2]{Setterqvist2005} obtained the estimate
\beq{EricLp}
N(p)\le C_{S}(p)=\left(\pi(p-1)\right)^{1/p'}\,\left(p(2-p)\right)^{1/p-1/2}
=C_H(p) \cdot A_p^2\cdot A_q.
\eeq

The maximum relative error of the estimate \eqref{EricLp} is about $3\%$. 
%
The following empirical approximation has absolute error less $10^{-3}$: 
$$
 C_S(p)-N(p)\approx \frac{(p-1)(2-p)}{8}.
$$

\subsection*{The numerical method}

In notation of Section~\ref{ssec:improving}, we have the equation $B_p^p \tilde f=f$ to solve. Its solution (if it exists) is also a solution of the equation $Bf=f$, which any maximizer must satisfy. In practice, we solve the equation $B_p^p f(x)=f(b-x)$, where both the function $f$ and the shift parameter $b$ are to be determined. 
We employ the direct iteration method defined by
\beq{iterations}
 f_n(x)=  (B^p_{p} f_{n-1})(b_n-x).
\eeq 
The shift parameter $b_n$ at each step is determined by the condition
$$
  \max_x f_n(x)=f_n(0)=1.
$$ 
The recurrence \eqref{iterations} implies that $\|f_n\|_p=1$ for all $n>1$. 
Due to Lemma~\ref{lemma:norm-increase}, the sequence of norms $\|k*f_n\|_{p'}$ is nondecreasing; due to the Young inequality, it is bounded; hence a limit exists.
The computation is stopped when $\|k*f_n\|_{p'}^{p'}-\|k*f_{n-1}\|_{p'}^{p'}<\eps$, where $\eps$ is the given tolerance.  We chose this criterion because we are not concerned (here) with computation of the solution $f$.

\smallskip
 The error of the numerical method has two sources besides the machine arithmetics and finiteness of the number of iterations.

(I) Domain compactification: the line $\RR$ is replaced by a finite interval $I=[-a,a]$ and the convolution on $\RR$ is replaced by the cyclic convolution on $I$.

(II)  Discretization: the functions of continuous variable are replaced by the functions of a discrete parameter.
We use the uniform grid with $N=2^k$ nodes.

\smallskip 
As we have noted in Remark~\ref{rem:symmetricB}, the existence of solution of the equation $B_p^p \tilde f=f$ has not been proved. This is not an important issue though: one can follow even-numbered iterations, since $\tilde B^2=B$ and the existence of solution of the equation $Bf=f$ is known.

The essential gaps in the justification of our numerical method
are the following: 

(a) a proof of convergence of the iterations \eqref{iterations} в $L_p(\RR)$ (even of the even-numbered iterations) is lacking;

(b) there is no result on uniqueness (up to a shift) of solution of the equation $Bf=f$, which means that the limit
$\lim_{n\to\infty} \|k*f_n\|_{p'}$ may in principle depend
on the initial condition.

\smallskip 
In practice, for a fixed compactification we observed a geometric convergence of the norms
$\|k*f_n\|_{p'}$
(and, moreover, a geometric convergence of $(f_n)$ in norm), the faster the closer $p$ is to $1$. The limit function
appears to be the same for different initial conditions. (We tried the initial conditions being either the Gaussians with various dispersions or the indicator functions of intervals.)  
In order to control the accuracy of the results, we performed computations with doubling of the number of nodes until stabilization. (In most cases, $N=512$ nodes were sufficient.) 
 
As regards the compactification, the Young inequality and the triangle inequality provide an upper bound for the error of the computed norm $\|K_k\|_{p,p}$ when the support of the kernel $k$ gets truncated. It is also easy to estimate the error due to the use of the cyclic convolution instead of the convolution on $\RR$. Contrary to the situation with convergence of iterations, the compactification appears to run into trouble as $p$ approaches 1, as the convolution kernel \eqref{LT-conv-kernel} becomes slowly decreasing in the negative directon. For instance, $h_{1.05}(y)\approx e^{-|y|/20}$ for $y<0$. However, when truncating the support of the function $h_p$, we are concerned not with absolute values of the cut-out, but with its $L_q$-norm. Since $q=p'/2$ and $h_p(y)^{q}=e^{y/2-(p'/2)\exp y}$,
the truncation parameter can be set uniformly in $p$. Computations support these considerations. For the purpose of control we used the doubling of the support $[-L,L]$ of the truncation of $h_p$ (and the corresponding doubling of the length of the circumference obtained by identifying the ends of the interval).   

\smallskip
All significant digits of the numerical data presented in Table~\ref{tab:LT-norm} and in the formula \eqref{minLaplaceNorm} are found to be stable with respect to the described operations of parameters doubling.  
 
\section{Open questions}
\label{sec:open-problems} 

\begin{opquest}
	\label{opq:uniqueness}
	Let $k\in L_q(\RR^d)$ and $f_1$, $f_2$ be two maximizers for
	the operator $K_k$ from $L_p$ to $L_{r'}$. 
	Is it true that there exist $\theta\in\RR$ and
	a vector $a\in\RR^d$ such that  $f_2(x)=e^{i\theta}f_1(x-a)$?
\end{opquest}
 
Having posed (and solved) the question of the existence of a minimizer, it is natural to ask about its uniqueness up to the trivial transformations. We suppose that in general there is no uniqueness. It looks probable however that one can formulate
conditions sufficient for uniqueness and embracing some narrow but meaningful class of functions $k$ (positive? unimodular?\dots). Exploring a finite-dimensional analog --- the convolution on $\ZZ/m\ZZ$ --- might help to understand what effects one should anticipate.  

In the case non-uniqueness is revealed, a number of further
questions can be asked, concerning non-maximizer
solutions of the equation $Bf=f$, bifurcation phenomena,
Morse indices etc. 

\begin{opquest}
\label{opq:q1p_not2}
Let $1<p<\infty$, $p\neq 2$. Does there exist a nonzero kernel
$k\in L_1$, for which the convolution operator $K_k:\,L_p\to L_p$ possesses a maximizer?
\end{opquest}

The affirmative answer in the case $p=2$ is given in Subsection~\ref{ssec:limit_cases} (subcase $\mathrm{I(B)_2}$).

\begin{opquest}
\label{opq:compact-perturbation}
	Let $k\in L_q$ and $V$ be a compact operator from $L_p$ to $L_{r'}$.
	Is it true that a maximizer for the operator $K_k+V$ exists? 
	Is this true in the particular case when $V$ is a rank one operator?
\end{opquest}

An answer to this question would yield either an extension of the class of admissible integral kernels in Theorem~\ref{thm:minimizer_exists} or an yet another counterexample,
in addition to the one given in Subsection~\ref{ssec:counterex-perturbation}, stressing the role of translation-invariance towards the existence of a maximizer. 

\begin{opquest}
	\label{opq:loc-compact}
	Generalize Theorem~\ref{thm:minimizer_exists} to embrace a certain class of locally compact groups (in particular, a (sub)class of discrete finitely generated groups).
\end{opquest}

The proof of Theorem~\ref{thm:minimizer_exists}
goes through in $\ZZ^d$ with trivial modifications.%
\footnote{The situation in the limit cases differs
between $\RR^d$ and $\ZZ^d$, cf.\ end of comment~\ereftxt{com:Tao-toy-problem}.}
One can try to get a clue about the case of discrete non-commutative groups by studying convolution on the free groups with two generators. One should be aware of the fact that
the condition \eqref{pqr} on the exponents in the Young inequality is, for a general locally-compact group, not necessary, cf.\ \cite{Quek-Yap_1983}. 

\begin{opquest}
\label{opq:maximizer_in_L1Linfty}
Investigate the local and global properties of a maximizer
as depending on the properties of the kernel $k$.
\end{opquest}

We have stated this question in a broad and imprecise form. 
Here are more specific sample questions, which we would be interested
to have answered.

\smallskip
{\sc Question}\ \arabic{opquest}A.
What is a guaranteed rate of decay of a maximizer
provided $k\in L_q$ has finite support? 

\smallskip
{\sc Question}\ \arabic{opquest}B.
What condition on $k$ (``a room of integrability'') beyond the assumed $k\in L_q$ is sufficient to guarantee that a maximizer lies in $L_\infty$?

\smallskip
According to Proposition~\ref{propos:maximizer_in_anyLp}, it is sufficient that $k\in L_{q+\eps}$ with arbitrarily small $\eps>0$. 
Isn't an ``inner room of integrability'' of the kernel $k$ (по with respect to $L_q$) and/or a maximizer (with respect to $L_p$) already sufficient? 
By an ``inner room of integrability'' we mean that, for example, $k$ belongs to some Orlicz space properly contained in $L_q$.
In the proof of Proposition~\ref{propos:maximizer_in_anyLp}, one can replace the reference to Young's inequality by the reference to O'Neil's inequality \cite{ONeil_1965} 
(concerning convoulution in Orlicz spaces), but it is unclear how far one can get with this approach.  

\begin{opquest} 
	\label{opq:dist-to-maximizer}
	{\bf Conjecture. } 
	For any $\eps>0$ there exists $\delta>0$ (depending on $q,p$ and the convolution kernel $k$)
	such that if $\|f\|_p=1$ and $\|k*f\|_{r'}> \|K\|_{p,r}(1-\delta)$, then $\inf_{g\in\mathfrak{M}_k}\|f-g\|_p<\eps$,
	where $\mathfrak{M}_k$ is the set of all maximizers for the operator $K$.
\end{opquest}

The formulation of the Conjecture parallels that of Proposition~\ref{propos:ineq-for-maximizer}, cf.\ comment \ereftxt{com:discrepancy}.

\begin{opquest}
	\label{opq:dist-to-Gaussians}
	Find a lower bound for the $L_p$-distance from the function $h_p$ defined by 
	\eqref{LT-conv-kernel} to the set of Gaussians, that is,
	estimate from below the quantity
	$$
	\delta_p=\inf_{r,b>0;m\in\RR}
	\norm{h_p(x)-re^{-b(x-m)^2}}_p.
	$$ 
\end{opquest}

This question is a step to the analytical improvement 
of the inequality \eqref{EricLp}: 
$$
 N(p)\le C_S(p)-c\delta_p\|h_p\|_q^3=C_S(p)-c\delta_p C_H(p)^3,
$$
where the constants $C_S(p)$ and $C_H(p)$ are from
\eqref{EricLp} and \eqref{HardyLp}, respectively, and
$c$ is the constant in the (one-dimensional) Christ inequality,
see comment \ereftxt{com:Christ} in Section~\ref{sec:comments}. 

\begin{opquest}
\label{opq:convergence-of-iterations}
Prove the convergence of the iterations \eqref{iterations}.
\end{opquest}
 
This question allows a broad interpretation (convergence of the iterations $f_{n+1}=Bf_n$ under some general assumptions), as well as a narrow interpretation: explain analytically why the iterations converge in the concrete situation of Section~\ref{sec:Laplace-norm}.

A potential non-uniqueness of solution of the equation $Bf=f$
(cf.\ Question~\ref{opq:uniqueness}) may call for certain adjustments of the question in its broad interpretation.

Note that in order to compute just the norm of the operator
$K_k$ (and not a maximizer) all that matters is not the convergence of iterations but exactly the absence of 
an extraneous solution $f_{wrong}$ with a small norm.

\section{Comments}
\label{sec:comments}

{\bf Section~\ref{sec:intro}}

\enote{com:Beckner} 
The inequality \eqref{beckC} has been proved independently and almost at the same time in \cite{Beckner1975b} and in
\cite{Brascamp-Lieb_1976}. See also the textbook \cite[Theorem~4.2]{Lieb-Loss_1998}. We note a simple proof given in \cite{Barthe_1997} 
(essintially based but on H\"older's inequality) and a particularly elegant proog in \cite{Carlen-Lieb-Loss_2004} (exploiting monotonicity of the trilinear form $(f_1*f_2,f_3)$ under heat equation evolution of the functions $f_i$).

A discussion of the Young inequality on locally compact groups with emphasis on admissible exponents and sharpness of the constants can be found in \cite{Fournier_1977, Quek-Yap_1983}.

\smallskip
\enote{com:Christ}
For review-style expositions of the results and methods of Christ's work \cite{Christ_2011,Christ_2014} we refer to \cite{Culiuc_2015}, \cite{Vitturi_2014}.  

The result particularly relevant to a possible improvement
of Setterqvist's estimate
\eqref{EricLp} is
\cite[Corollary~1.5]{Christ_2014}:

{\it Let $f_j\in L_j(\RR^d)$
($j=1,2,3$) and
$\|f_1\|_{p_1}=\|f_2\|_{p_2}=\|f_3\|_{p_3}=1$. 
Put
$C=(A_{p_1}A_{p_2}A_{p_3})^d$
(which is Beckner's constant in $\RR^d$)
and denote $\mathfrak{G}$ the set of all Gaussian functions,
$$
\mathfrak{G}=\left\{\phi(x)=r e^{-(Bx,x)+(a,x)}\right\}, 
$$
where
$r>0$, $a\in\RR^d$ and $B$ is a positive definite quadratic form.
There exists a constant $c>0$ (which depends on the dimension $d$) such that
$$
|(f_1*f_2,f_3)|\le C-\eps(\delta),
\qquad \eps(\delta)=c\delta^4,
$$
where 
$$
\delta=\inf_{g_1,g_2,g_3\in\mathfrak{G}}
\max_{j\in\{1,2,3\}} \|f_j-g_j\|_{p_j}.
$$
}

In order to use the  stated result for improvement of the estimate \eqref{EricLp}, one needs the numerical value of the constant $c$ (for $d=1$), which is not given in \cite{Christ_2014}, as well as
a lower estimate for the $L_p$-distance from the
kernel $h_p$ to the set of Gaussians. We offer the latter calcluation as an open question, see Question\ref{opq:dist-to-Gaussians} in Section~\ref{sec:open-problems}.

\smallskip
\enote{com:Lions-intro-problem}
In \cite[\S\,I.1]{Lions_1984a} the minimization problem for the functional of the form  
$$
\mathcal{E}(u)=\int_{\RR^n} e(x,Au(x))\,dx
$$
under the constraint  $\mathcal{J}(u)=\lambda$
is considered. 
Here
$$
\mathcal{J}(u)=\int_{\RR^n} j(x,Bu(x))\,dx;
$$
$e(\cdot,\cdot)$, $j(\cdot,\cdot)$ are given real-valued functions, $j\ge 0$; $u(\cdot)$ are elements of a given function space on $\RR^n$. Denote
$$
I_\lambda=\inf_{\mathcal{J}(u)=\lambda}\mathcal{E}(u). 
$$
A particular case with $x$-independent functions $e$, $j$
is referred to, in the general context, as ``problems at infinity''.

Our problem concerning the norm of the operator $K_k$
corresponds to 
$$
\mathcal{E}(u)=-\|u*k\|_{r'}^{r'},
\qquad 
\mathcal{J}(u)=\|u\|_p^p, 
$$
i.e.\ $e(u)=|u|^{r'}$, $Au=u*k$,
$Bu=u$, $j(u)=|u|^p$. The lower bound then is  
\beq{Lions-min-conv}
I_\lambda=-C\lambda^{\gamma},
\quad \gamma=\frac{r'}{p},\quad C=\|K_k\|_{p,r}^{r'}.
\eeq

What Lions' method provides is not a single general theorem but a general approach to proving the existence of extermizers in
a broad class of variational problems of analysis and mathematical physics.
It contains heuristic elements, so that details may vary and require a concrete, problem-specific approach.

\smallskip
The monograph \cite{Tintarev-Fieseler_2007} treats many aspects of the concentration compactness method, with emphasis on
convergence in Hilbert (Sobolev) spaces.
 A Russian-language reader may find Ch.~5 in the textbook \cite{Korpusov-Sveshnikov_2011} as a useful reference concerning Lions' method.

\smallskip
By all indications, it should be possible to prove  
Theorem~\ref{thm:minimizer_exists}\ in the framework of Lions' method; however, this would be a separate and not quite trivial project. Note that our proof neither refers to the Concentration Compactness Lemma \cite[Lemma\,I.1]{Lions_1984a}, \cite[Lemma~5.1]{Korpusov-Sveshnikov_2011} nor contains its close analog; the variants of ``vanishing'' and ``dichotomy'' are implicitly eliminated by other means. 

\smallskip
\enote{com:Tao-toy-problem}
In T.\,Tao's methodical article \cite[\S~1.6]{Tao_2009b}, a technique of ``profile decomposition''
is discussed: a ``profile'' (a function sequence) is  decomposed into a sum of shifts of fixed functions and a relatively compact sequence (cf.\ \cite[\S\,3.3, Theorem~3.1]{Tintarev-Fieseler_2007}). As an application, a ``toy theorem'' is proved, asserting that the discrete convolution operator acting from $\ell^1=L_1(\ZZ)$
to $\ell^p$ by the formula $(Kf)_n=f_n-f_{n-1}$ has a maximizer.  Note that in the corresponding case  II(A) of Subsection~\ref{ssec:limit_cases}
a maximizer in $\RR$ does not exist.  
%

\medskip
{\bf Section~\ref{sec:proof-general}}
%
%

\enote{com:def-tight} 
Some notions (the ``$\delta$-near'' ones) introduced in Definitions~\ref{def:delta-support}--\ref{def:concentration}
are there just to suit our local purposes, while other notions, with their origins in Probability Theory, have been used in different contexts. Among the latter, the term \emph{tight} is standard, cf.\ e.g.\ \cite[v.~2, \S\,8.6]{Bogachev_2003}. Let us comment on the remaining two.

%
%

\smallskip
1. The $\delta$-diameter introduced in Definition~\ref{def:delta-support} is L\'evy's dispersion function 
\cite[Section~1.1, Supplement~4]{HT_Concentration_functions_1980} in disguise. Specifically,
for a fixed $f\in L_p(\RR^d)$ with unit $p$-norm and a unit vector $v\in\RR^d$, we have the distribution function in the sense of probability theory
$$
F(t)=\int_{(v,x)<t} |f(x)|^p\,dx.
$$
The corresponding 
\emph{L\'evy concentration function}\ \cite{HT_Concentration_functions_1980} is
$$
Q_F(\lambda)=\sup_{t\in\RR} (F(t+\lambda)-F(t)).
$$ 
The inverse function is known as \emph{the dispersion function for the measure $dF$};
in our notation it is
$$
D^p_{\delta,v}(f)=\inf_{Q_F(\lambda)\ge 1-\delta} \lambda.
$$

2. We thought it useful to have a shorter name for the property
of a function sequence to be \emph{tight up to translations}; we call such a sequence \emph{relatively tight}. P.\,L.~Lions, in the formulation of his Concentration Compactness Lemma \cite[Lemma~I.1]{Lions_1984a}, chose to characterize the said property as the ``case of compactness'' rather than to devise a descriptive adjective.


\smallskip
\enote{com:delta-support}
We do not claim uniqueness of the $\delta$-near-support, 
cf.\ \cite[\S\,1.1.2]{HT_Concentration_functions_1980}.

\medskip
{\bf Section~\ref{sec:non-spreading}}

\enote{com:subadditivity_lemma}
The inequality of Lemma~\ref{lemma:u} is interpreted in Lions' theory as the subadditivity property (of crucial importance) of the fucntion $I_\lambda$, cf.\ \eqref{Lions-min-conv}. 
%
Note that the exponent $\gamma$ in the final application of Lemma~\ref{lemma:u} (see Lemma~\ref{lemma:concentration_property})
coincides with that in \eqref{Lions-min-conv}.

\smallskip
\enote{com:midpoint-average_lemma}
The function in the left-hand side of the inequality of Lemma~\ref{lemma:midpoint-average} is known as Steklov's averaging of the function $|f|$; the lemma states one of its most elementary properties.
The quantifiers can be swapped (``there exists $t_0$ such that for any $a$\dots'')
at the expense of putting an appropriate constant in the numerator of the right-hand side; this follows from the Hardy-Littlewood maximal inequality.

\smallskip
\enote{com:prevent_splitting}    
Perhaps, this place in our proof --- the reference to 
Lemma~\ref{lemma:h1+h2} in the proof of Lemma~\ref{lemma:controlled_support} --- most closely corresponds to 
Lions' thesis
``prevent the possible splitting of minimizing sequences by keeping them concentrated'' \cite[p.114]{Lions_1984a}, and also
reflects the ``asymptotic orthogonality'' phenomenon 
\cite{Tao_2009b}.



\medskip
{\bf Subsection~\ref{ssec:compact}}

\smallskip
\enote{com:compactness-lemma}
For general integral operators in $L_p$ spaces, sufficient conditions for compactness usually require some spare room in the space exponents as compared with sufficient conditions for boundedness, cf.\ e.g.\ 
\cite[Theorem~7.1]{Krasnoselski_1966}.%
\footnote{
Note that the usual notation $L_p$ (or $L^p$) corresponds to $L_{1/p}$ in \cite{Krasnoselski_1966}.}
As it is readily seen, there is no such ``spare room'' in the conditions of Lemma~\ref{lemma:converge-convolutions}.

A very general study of compositions of convolution and multiplication operators in Lebesgue spaces is found in the paper \cite{Busby-Smith_1981}.
Our Lemma~\ref{lemma:converge-convolutions} is a particular  case of Theorem~6.4 of \cite{Busby-Smith_1981}; however,
it seems easier to give an independent proof, as we did,
than to scrutinize involved notation and conditions.

\medskip
{\bf Subsection~\ref{ssec:discussion-survey}}

\enote{com:discrepancy} 
It is harder (likely, much harder) to prove, under the assumptions of Proposition~\ref{propos:ineq-for-maximizer}, the approximative property in the spirit of M.~Christ's results
mentioned in the comment \ereftxt{com:Christ}.
Cf.\ Question~\ref{opq:dist-to-maximizer} in Section~\ref{sec:open-problems}.

\smallskip
\enote{com:Lions-convergence}
One of general heuristic principles stated by P.\,L.~Lions reads
``All minimizing sequences are relatively compact up to a translation iff [a certain] strict subadditivity inequality holds''.
\cite[p.\,114]{Lions_1984a}
We took it as a hint that what is now
Proposition~\ref{prop:any-subseq-conv} should be valid, although we did not need it in the proof of Theorem~\ref{thm:minimizer_exists}. 

Adapting a notion of \emph{shift-compactness}\ \cite[Section~5.1.1]{HT_Concentration_functions_1980} to our situation,
the short summary can be stated: {\it under the assumptions of Theorem~\ref{thm:minimizer_exists}, every maximizing sequence of the convolution operator is shift-compact}.

\smallskip
\enote{com:convolution-lower-bound}
The result of Subsection~\ref{ssec:norm-lower-bound} 
is complementary to the results of the paper
\cite{Nursultanov-Saidahmetov_02}, in which the norms 
$\|K_k\|_{p,r}$ are estimated from below in terms of (absolute values of) the integrals of the kernel over certain families of
sets, which are different for positive kernels and general real-valued kernels.

\medskip
{\bf Subsection~\ref{ssec:necessary-cond-extremum}}

\enote{com:rev-Holder}
Another approximative version of H\"older's inequality, with explicit constants, is found in
\cite{Aldaz_08}.


\end{document}